\documentclass[11pt]{amsart}
\usepackage{amssymb, latexsym}
\theoremstyle{plain}
\newtheorem{theorem}{Theorem}
\newtheorem{corollary}{Corollary}

\newtheorem{proposition}{Proposition}

\newtheorem*{2'}{Theorem 2'}
\newtheorem*{3'}{Theorem 3'}

\theoremstyle{remark}

\newtheorem*{Remark 1}{Remark 1}
\newtheorem*{Remark 2}{Remark 2}
\newtheorem*{Remark 3}{Remark 3}
\newtheorem*{Remark 4}{Remark 4}
\newtheorem*{Clemma}{\bf Combinatorial Lemma}

\numberwithin{equation}{section}

\begin{document}

\title [A View From   the Bridge Spanning Combinatorics and Probability]
{A View From the Bridge  Spanning Combinatorics and Probability}

\author{Ross G. Pinsky}

%\noindent  pinsky@math.technion.ac.il\ \ \ \ tel: 972-4-829-4083\ \ \  fax: 972-4-829-3388

\address{Department of Mathematics\\
Technion---Israel Institute of Technology\\
Haifa, 32000\\ Israel}
\email{ pinsky@math.technion.ac.il}

\urladdr{http://www.math.technion.ac.il/~pinsky/}

\date{}

\begin{abstract}
This paper presents an offering of some of the myriad connections between Combinatorics and Probability, directed in particular toward combinatorialists.
% Most of the results presented here are accompanied by complete accessible proofs, and for   most of the  remaining ones
% detailed sketches of the proofs are supplied.
 The choice of material was dictated by the author's own interests, tastes and
 familiarity, as well as by
 a desire to present  results with either  complete proofs or
well developed sketches of proofs, and to ensure
that the arguments  are rather accessible to combinatorialists.
The first several sections collect some concepts and rudimentary results from probability theory
that are needed to understand the rest of the paper.
\end{abstract}

\maketitle

\tableofcontents

\section{Introduction}\label{intro}
I was approached by the editors to write a survey article, accessible to combinatorialists, discussing the ``bridge'' between  combinatorics and probability.
The oldest lanes  on that bridge are more than three hundred years old.
Probability problems from that period and for quite some time afterwards as well, involving cards and dice, helped spur the development of basic combinatorial principles, which were then applied to solve the probability problems.
 The initial lanes constructed in the other direction might be said to exist almost as long; indeed, de Moivre's seminal 1733 paper proving the Central Limit Theorem in the context of fair coin-tossing may be thought of as
a way of estimating binomial coefficients through a probabilistic lens. A very different aspect of probability
in the service of combinatorics
 dates to the 1940's when  the so-called ``Probabilistic Method'' was first popularized by Erd\H{o}s.
Around the same time,  an interest in random permutations began to develop among certain probabilists.
In this latter case, it is difficult to say which field is the giver and which the receiver.
 The rising tide  of graph theory and  computer algorithms
over the past seventy years has certainly  provided fertile ground for the use of probability in combinatorics.
By today this bridge between combinatorics and probability
has amassed  a multitude of decks and lanes, far more than can be dealt with in a  single survey  of modest length.
Thus, the choice of material was dictated by my own interests, tastes and
 familiarity.
It was also   guided  by a desire to present  results with either  complete proofs or
well developed sketches of proofs, and to ensure
that the arguments  are rather accessible to combinatorialists, including graduate students in the field.
  A large majority of the material herein concerns enumerative combinatorics.
For more on the  historical view of the connection between combinatorics and probability, see
 \cite{DB,Feller1, Rior}.

Elementary probability based on equally likely outcomes is almost by definition co-extensive with
enumerative combinatorics. Today, probabilists sometimes study combinatorial structures on which  probability
distributions other than the  uniform
ones are used. These other distributions are usually constructed as an exponential family with a parameter.
For example, in the context of random permutations, a lot of work has
 been done in recent years in the case that the uniform measure is replaced by a measure from the family of  Ewens sampling
distributions or from the family of  Mallows distributions. For the Ewens sampling distributions, see for example
\cite{ABT, Fer}; for the Mallows distributions, see for example  \cite{BP,GP,GO,MS, Pin21}.
In this survey, we restrict to the uniform distribution on combinatorial objects.

 I now note several from among many of the important and interesting directions that I do not touch upon  at all.   One such direction is Analytic Combinatorics, where probability
 and complex function theory unite  to provide answers to problems concerning the probabilistic structure of large combinatorial objects. We refer the reader
 to the book \cite{FS} for a comprehensive treatment and many additional references.
 See also the book \cite{PW}.
 Another such direction is the problem of the length of the longest increasing subsequence in a random permutation. This
 very specific problem turns out to be packed with a
 great  wealth  and variety of mathematical ideas and directions.
 The original seedling for this direction was provided in \cite{H}. Then fundamental contributions were
 provided in \cite{VK,LS}, followed by the tour-de-force \cite{BDJ}.
  We refer the reader to the book \cite{Rom}
 for a very readable treatment  and comprehensive references.
Another direction not touched upon here is random trees. We refer the
reader to the book \cite{D} and the survey article \cite{LeG}, along with references therein, to  the classical articles \cite{Al1,Al2}  on
the Brownian continuum random tree, and for example to the papers \cite{BM,MM, LeGM}.
Yet another direction not discussed in this article  is the probabilistic aspects of pattern avoiding permutations.
See for example \cite{HRS1,HRS2,J1,J2,J3} and references therein.
For the combinatorial aspects of pattern avoiding permutations,  see for example \cite{Bona}.

A nice set of lecture notes on combinatorics and probability is provided by \cite{Pitman06}.
The paper \cite{Pitman01-02} aims to bridge several results in combinatorics and probability.
The book \cite{Sachkov} concerns   probabilistic methods in solving combinatorial problems.
A combinatorics book with many probabilistic exercises is \cite{Co}.

To access the bridge on the probability side, one must pass a toll booth, which consists of some basic probability.
 This is the material in  sections \ref{concepts}, \ref{distributions} and \ref{tools}. An excellent reference for this material is \cite{Durrett}; see also
 \cite{Feller1}.
Readers who so wish   can skirt the toll booth and access the bridge at section
\ref{PST}, referring back if and  when necessary .

 \section{Some basic probabilistic concepts}\label{concepts}

Let $\Omega$ be a set,  $\mathcal{F}$  a $\sigma$-algebra on $\Omega$ and  $P$ a probability measure on
the measure space $(\Omega,\mathcal{F})$. The triplet $(\Omega, \mathcal{F},P)$ is called a probability space.
In the case that $\Omega$ is finite or countably infinite, there is no reason not to choose $\mathcal{F}$ to be the $\sigma$-algebra consisting of all subsets of $\Omega$.
In such cases, we will refer to $(\Omega,P)$ as a probability space.
A  random variable
$X$ on the probability space $(\Omega, \mathcal{F},P)$ is a measurable function from $(\Omega,\mathcal{F})$ to $(\mathbb{R},\mathcal{B})$, where
$\mathcal{B}$ is the Borel $\sigma$-algebra on $\mathbb{R}$.
 For $A\in\mathbb{B}$, we write $P(X\in A)$ for $P(X^{-1}(A))$.
 The expected value of $X$ is defined as
 $$
 EX=\int_{\Omega}XdP,
 $$
  if $X$ is integrable (that is, if $\int_\Omega|X|dP<\infty$) and if so, then  the variance  of $X$ is defined
as
$$
\sigma^2(X)=E(X-EX)^2=\int_\Omega(X-EX)^2dP,
$$
 which
 can be written as
 $$
 \sigma^2(X)=EX^2-(EX)^2.
 $$
 For   $k\in\mathbb{N}$, the $k$th moment of the random variable is defined as
 $$
 EX^k=\int_\Omega X^kdP,
 $$
 if $\int_\Omega|X|^kdP<\infty$. If $\int_\Omega|X|^kdP=\infty$, we say that the $k$th moment does not exist.
A subset $D\subset \Omega$ is called an event.
A collection $\mathcal{D}$ of events  is called independent
if for every finite or countably infinite collection
  $\{D_j\}_{j\in I}\subset\mathcal{D}$, one has $P(\cap_{j\in I}D_j)=\prod_{j\in I}P(D_j)$.
  A sequence $\{X_n\}_{n=1}^\infty$ of  random variables is called independent if
  $P(X_1\in A_1,X_2\in A_2,\cdots)=\prod_{n=1}^\infty P(X_n\in A_n)$, for  all Borel sets $A_n\subset\mathbb{R},\ n=1,2,\cdots$.

Given events $A,B$ with $P(A)>0$, the conditional probability of the event $B$ given the event $A$ is denoted by $P(B|A)$ and defined
by
$$
P(B|A)=\frac{P(A\cap B)}{P(A)}.
$$

The distribution $\nu$ of a random variable $X$ is the probability measure on $\mathbb{R}$
defined $\nu(A)=P(X\in A)$, for Borel sets $A\in\mathcal{B}$.
The distribution function  $F(x)$ of a random variable $X$ is defined by
$$
F(x)=P(X\le x),\ x\in\mathbb{R}.
$$
It follows that $F$ is continuous at $x$ if and only if $x$ is not an atom for the distribution $\nu$ of $X$; that is if and only if $\nu(\{x\})=P(X=x)=0$.
Given a probability measure $\nu$ on $\mathbb{R}$, one can always
construct  a probability space and a random variable on it whose distribution is $\nu$.

A sequence $\{\nu_n\}_{n=1}^\infty$ of probability measures on $\mathbb{R}$ is said to converge weakly to the probability measure $\nu$ on $\mathbb{R}$  if $\lim_{n\to\infty}\nu_n(A)=\nu(A)$, for all Borel sets $A\in\mathcal{B}$ which satisfy
$\nu(\partial A)=0$, where $\partial A$ denotes the boundary of $A$.
 With the help of some nontrivial measure-theoretic machinery, weak convergence
can be shown  to  be equivalent to the seemingly weaker requirement that
replaces the class of all Borel sets above by the class of rays $(-\infty,x]$, with $x\in\mathbb{R}$.

 A sequence $\{X_n\}_{n=1}^\infty$ of random variables is said to converge in  distribution to a random variable
  $X$  if their distributions $\{\nu_n\}_{n=1}^\infty$ converge weakly to the distribution $\nu$ of $X$.
In light of the  above two paragraphs, this requirement can be written as
$$
\lim_{n\to\infty}F_n(x)=F(x),\ \text{for all}\ x\in \mathbb{R}\ \text{at which }\ F\ \text{is continuous}.
$$
%(We have written the same $P$ for all the random variables above, but in fact this is not necessary; indeed, each random %variable could be defined on its own probability space.)
We write $X_n\stackrel{\text{dist}}{\to}X$.
Note that if $X$ is almost surely equal to the constant $c$, then $X_n\stackrel{\text{dist}}{\to}X$
is equivalent to
$\lim_{n\to\infty}P(|X_n-c|\ge\epsilon)=0$, for every $\epsilon>0$.

\section{The binomial distribution, the Poisson approximation and some other classical distributions}\label{distributions}
For each $n\in\mathbb{N}$ and each $p\in[0,1]$, the \it binomial distribution\rm\ $\text{Binomial}(n,p)$ describes the
random number $X_{n,p}$ of ``successes'' in a sequence of  independent trials, where for each trial the probability of success
is $p$. Such trials are known as \it Bernoulli trials.\rm\
Simple probabilistic and combinatoric reasoning reveals that
$$
P(X_{n,p}=j)=\binom njp^j(1-p)^{n-j},\  j=0,1,\cdots, n.
$$

A random variable $X_\lambda$ has the \it Poisson distribution\rm\ with parameter $\lambda>0$
  if
$$
P(X_\lambda=j)=e^{-\lambda}\frac{\lambda^j}{j!},\ j=0,1\cdots.
$$
We denote this distribution by $\text{Poisson}(\lambda)$.
Unlike the case of the binomial distribution, at first glance
one does not see any combinatorics here; yet as we shall see directly below and in later sections,
quite a bit of combinatorics is  lurking behind the scenes.
A basic result appearing in a first course in probability theory is the ``Poisson approximation to the binomial
distribution,'' which states that for large $N$ and small $p$, the  Poisson$(Np)$ distribution approximates the Binomial$(N,p)$ distribution.
More precisely,
\begin{equation*}
\lim_{N\to\infty, Np\to\lambda}P(X_{N,p}=j)=P(X_\lambda=j), \ j=0,1,\cdots;
\end{equation*}
that is,
\begin{equation}\label{PAB}
\lim_{N\to\infty,Np\to\lambda}\binom Njp^j(1-p)^{N-j}=e^{-\lambda}\frac{\lambda^j}{j!}, \ j=0,1,\cdots.
\end{equation}
The proof of \eqref{PAB} is an easy calculus exercise. We note that \eqref{PAB} is equivalent to  the statement
$$
X_{n,p_n}\stackrel{\text{dist}}{\to}\ X_\lambda\
\text{as}\ n\to\infty \ \text{and}\  np_n \to\lambda.
$$

The Poisson distribution arises in a wealth of applications, the most canonical of which is probably
with regard to call centers. Under many real-life situations, the number of calls that arrive at a call center
during $t$ units of time, for any $t>0$, can be modeled by the Poisson distribution with parameter $\lambda t$,
where $\lambda$ is referred to as the intensity parameter. The intuition behind this is as follows.
Fix $t>0$, and let $n\in\mathbb{N}$. Divide the time interval $[0,t]$ into subintervals $\{J^{(n)}_i\}_{i=1}^n$
all of length $\frac tn$. One
makes two assumptions: (i) the (random) number of calls in different subintervals are independent and identically distributed; (ii)
in each subinterval $J^{(n)}_i$,  the probability of
at least one call is $\frac{\lambda t} n+o(\frac1n)$, for some appropriate $\lambda>0$ that characterizes the call center,
and the probability of more than one call is $o(\frac1n)$.
Under this scenario, it is easy to see that the distribution of the  number of calls can be approximated by
Bin$(n,\frac{\lambda t} n)$. (Note that the probability that no interval has more than one call is given by
$(1-o(\frac1n))^n$, which approaches one as $n\to\infty$.)
And by the Poisson approximation, Binomial$(n,\frac{\lambda t} n)$ is close to the  Poisson$(\lambda t)$ distribution.
For basic applications of the Poisson approximation, see \cite{Feller1}, and for a more extensive study, see
\cite{BHJ}.

A random variable $X$ has the \it geometric distribution\rm\ with  parameter $\rho\in(0,1)$ if
 $P(X=k)=\rho^k(1-\rho)$, $k=0,1,\cdots$.

So far all the distributions we've presented have been discrete; that is there exists a finite or countably infinite set $\{x_j\}_{j\in I}$,
such that a random variable $X$  with the distribution in question satisfies $P(X=x_j)>0$, for all $j\in I$, and $\sum_{j\in I}P(X=x_j)=1$. Now we recall three continuous distributions that will appear in this survey. A distribution is continuous if a random variable with the distribution
in question satisfies $P(X=x)=0$, for all $x\in\mathbb{R}$; equivalently, the random variable's distribution function $F(x)$ is a continuous function.

A random variable $X$ has the  \it exponential distribution\rm\ with parameter $\lambda>0$ if its distribution function
is given by $F(x)=1-e^{-\lambda x}$, for $x\ge0$. Its  density function is given by $f(x):=F'(x)=\lambda e^{-\lambda x},\ x\ge0$.
A random variable has the \it standard normal\rm\ (or \it Gaussian\rm) distribution, denoted by $N(0,1)$, if its distribution function
is given by $F(x)=\int_{-\infty}^x\frac{e^{-\frac{y^2}2}}{\sqrt{2\pi}}dy$, $x\in\mathbb{R}$.
Its density function is
$f(x)=\frac{e^{-\frac{x^2}2}}{\sqrt{2\pi}}$, $x\in\mathbb{R}$.
For $X_{N(0,1)}$  a random variable with the standard normal distribution, one has $EX_{N(0,1)}=0$ and
$\sigma^2(X_{N(0,1)})=1$.
A random variable $U$ has the \it uniform distribution \rm on $[0,1]$ if
its distribution function satisfies $F(x)=x$, for $x\in[0,1]$. Its density
function is  $f(x)=1$, $x\in[0,1]$.

\section{Some basic probabilistic tools and results}\label{tools}
Consider the probability space $(\Omega,\mathcal{F},P)$.
For events $\{D_j\}_{j=1}^N$, we have by the subadditivity of measures,
$$
P(\cup_{j=1}^ND_j)\le \sum_{j=1}^NP(D_j),
$$
with equality if and only if $P(D_i\cap D_j)=0$, for $i\neq j$.

Let $X$ be a random variable,  denote its expected value by $EX$ and
let $\sigma^2(X)=E(X-EX)^2$ denote its variance.
\it Markov's inequality\rm\ states that
$$
P(|X|\ge L)\le \frac1L E|X|,\ \text{for}\ L>0.
$$
 Multiplying both sides of the above inequality by $L$, one sees that its proof is essentially immediate from the definition of the expectation.
We recall \it Chebyshev's inequality\rm, an extremely simple but often very useful formula.
It states that
$$
P(|X-EX|\ge\lambda)\le \frac{\sigma^2(X)}{\lambda^2},\  \lambda>0.
$$
Here is a one-line proof using Markov's inequality (with $L=1$):
$$
P(|X-EX|\ge\lambda)=P\left(\frac{(X-EX)^2}{\lambda^2}\ge1\right)\le E\frac{(X-EX)^2}{\lambda^2}= \frac{\sigma^2(X)}{\lambda^2}.
$$

Let $\{X_n\}_{n=1}^\infty$ be a sequence of
random variables and let $S_n=\sum_{j=1}^nX_j$. We have
$$
ES_n=\sum_{j=1}^nEX_j.
$$
 If the random variables are independent and have finite variances, then a rather straightforward calculation reveals that
the variance of the sum equals the sum of the variances:
$$
\sigma^2(S_n)=\sum_{j=1}^n\sigma^2(X_j).
$$
Now assume that $\{X_n\}_{n=1}^\infty$ are independent and identically distributed (henceforth IID).
Assuming that the expected value exists for these random variables,
and denoting it by $\mu$,
the \it Weak Law of Large Numbers \rm\ holds; namely,
$$
\frac{S_n}n\stackrel{\text{dist}}{\to}\mu.
$$
As noted earlier, the above result is equivalent to
\begin{equation}\label{WLtoconst}
\lim_{n\to\infty}P\left(|\frac{S_n}n-\mu|\ge \epsilon\right)=0,\ \text{for all}\ \epsilon>0.
\end{equation}
Under the additional assumption of finite variance,  \eqref{WLtoconst} follows immediately by applying Chebyshev's inequality to the random variable $\frac{S_n}n$,
and using the fact that $\sigma^2(S_n)=n\sigma^2(X_1)$.

Assuming  that the IID random variables $\{X_n\}_{n=1}^\infty$ have a finite variance $\sigma^2$, then
the \it Central Limit Theorem\rm\ holds. This result describes the probabilistic fluctuations of $\frac{S_n}n$ from
it expected value $\mu$:
$$
\frac{\sqrt n}\sigma\big(\frac{S_n}n-\mu\big)=\frac{S_n-n\mu}{\sigma\sqrt n}\stackrel{\text{dist}}{\to} X_{N(0,1)},
$$
where $X_{N(0,1)}$ denotes a random variable with the standard normal distribution.

More generally, if $\{W_n\}_{n=1}^\infty$ is an arbitrary sequence of random variables
with finite expectations,   we say that the weak law of large numbers holds for
$\{W_n\}_{n=1}^\infty$ if
$$
\frac{W_n}{EW_n}\stackrel{\text{dist}}{\to}1,
$$
 and
if furthermore, the variances $\sigma^2(W_n)$ are finite, we say that the central limit theorem holds for
$\{W_n\}_{n=1}^\infty$
if
$$
\frac{W_n-EW_n}{\sigma(W_n)}\stackrel{\text{dist}}{\to} X_{N(0,1)}.
$$

Let        $\{W_n\}_{n=1}^\infty$ be an arbitrary sequence of  random variables.
The \it first moment method\rm\ states that:
\begin{equation}\label{firstmommeth}
\text{if} \
\lim_{n\to\infty}E|W_n|=0,\ \text{then}\ W_n\stackrel{\text{dist}}{\to}0.
\end{equation}
The proof is immediate from Markov's inequality; indeed,
$$
P(|W_n|\ge\epsilon)\le \frac1\epsilon E|W_n|\stackrel{n\to\infty}{\to}0,\ \text{for all}\ \epsilon>0.
$$
The \it second moment method\rm\ states that
\begin{equation}\label{secondmommeth}
\text{if}\ \sigma^2(W_n)=o((EW_n)^2),\ \text{then}\
\frac{W_n}{EW_n}\stackrel{\text{dist}}{\to}1.
\end{equation}
The proof is almost immediate from Chebyshev's inequality.  Indeed, let $Y_n=\frac{W_n}{EW_n}$. Then  $EY_n=1$
and $\sigma^2(Y_n)=\frac1{(EW_n)^2}\sigma^2(W_n)$, where we use the fact that
$\sigma^2(cX)=c^2\sigma^2(X)$, if $X$ is a random variable and $c$ is a constant.
Thus, by Chebyshev's inequality, for any $\epsilon>0$,
$$
P(|\frac{W_n}{EW_n}-1|\ge\epsilon)=P(|Y_n-EY_n|\ge\epsilon)\le\frac{\sigma^2(Y_n)}{\epsilon^2}=
\frac{\sigma^2(W_n)}{\epsilon^2(EW_n)^2}\stackrel{n\to\infty}{\to}0.
$$

A frequently employed tool in these pages is the method of indicator random variables.
The indicator random variable $I_A$ for an event $A\subset \Omega$ is defined by
$$
1_A(\omega)=
\begin{cases} 1,\ \text{if}\ \omega\in A;\\0,\ \text{if}\ \omega\in\Omega- A.\end{cases}
$$
 We  illustrate this tool first with the case of a binomial random variable.
According to the definition, we have
$EX_{n,p}=\sum_{j=0}^n j\binom nj p^j(1-p)^{n-j}$.
A simpler way to calculate this expectation is to recall the probabilistic model that gave rise to $X_{n,p}$; namely that
$X_{n,p}$ is the number of successes in a sequence of $n$ independent Bernoulli trials, where for each trial the probability of success is $p$.
Thus, we define the indicator random variable $I_i$ to be equal to 1 if there was a success on the $i$th Bernoulli trial and 0 if there was a failure.
Then $I_i$ has  the following \it Bernoulli distribution\rm\ with parameter $p$: $P(I_i=1)=1-P(I_i=0)=p$.
Clearly $EI_i=p$. We can represent $X_{n,,p}$ as
$$
X_{n,p}=\sum_{i=1}^nI_i,
$$
and then by the linearity of the expectation, we immediately obtain
$$
EX_{n,p}=np.
$$
Note that this calculation does not use the fact that the $\{I_i\}_{i=1}^n$ are independent.
An easy calculation shows that the variance of  $I_i$  is equal to $p(1-p)$.
Thus, we have
$$
\sigma^2(X_{n,p})=\sum_{i=1}^n\sigma^2(I_i)=np(1-p).
$$

When we employ  indicator random variables in this article, the typical situation will be
$$
X=\sum_{i=1}^n 1_{A_i},
$$
where the events $\{A_i\}_{i=1}^n$ are not independent.
Since $E1_A=P(A)$,  we have
$$
EX=\sum_{i=1}^nP(A_i).
$$
Since the variance is given by
 $\sigma^2(X)=EX^2-(EX)^2$, we will have the variance at our disposal as soon as we calculate the second moment.
  We have
  \begin{equation}\label{variancesumindicator}
EX^2=E(\sum_{i=1}^n1_{A_i})^2=\sum_{i=1}^n P(A_i)+2\sum_{1\le i< j\le n}P(A_i\cap A_j).
\end{equation}

We conclude this section with the method of moments for proving convergence in distribution \cite{Durrett}.
Consider a random variable $X$ which possesses all of its moments.
Under the growth condition
$$
\limsup_{k\to\infty}\frac1k(EX^{2k})^{\frac1{2k}}<\infty,
$$
 the convergence in distribution
of a sequence of random variables  $\{X_n\}_{n=1}^\infty$ to $X$ follows if
\begin{equation*}\label{momentmethod}
\lim_{n\to\infty}EX_n^k=EX^k,\ \text{for all}\ k\in\mathbb{N}.
\end{equation*}

\section{Moments of the Poisson distribution, Stirling numbers of the second kind and Touchard polynomials}\label{PST}

It turns out that the Poisson distribution has an intimate connection with the \it Stirling numbers of the second kind\rm\ $S(n,k)$, $n\in\mathbb{N}, 1\le k\le n$.
Recall that  $S(n,k)$   counts the number of ways to partition a set of $n$ labelled objects into $k$ non-empty subsets
\cite{Man1,Stan1}. For convenience in formulas, one defines $S(n,0)=0$, for $n\in\mathbb{N}$ and $S(0,0)=1$.
The connection comes via the moments of the Poisson distribution. The $n$th moment, by definition,   is given by
\begin{equation}\label{nthmom}
EX_\lambda^n=\sum_{k=0}^\infty (e^{-\lambda}\frac{\lambda^k}{k!})k^n ,\ n=0,1,\cdots;
\end{equation}
however we will calculate the moments using a different tack.

It will be convenient to define the falling
factorial:
$$
(x)_j=x(x-1)\cdots (x-j+1)=j!\binom xj,\
\text{for}\ j\in\mathbb{N}\ \text{and}\ x\in\mathbb{R}.
$$
 Also,   for convenience, one defines $(x)_0=1$.
The falling factorials and the Stirling numbers of the second kind are connected by the formula
\begin{equation}\label{Stirlingfalling}
x^n=\sum_{j=0}^nS(n,j)(x)_j,\ n=0,1,\cdots.
\end{equation}
We supply a quick proof of \eqref{Stirlingfalling} \cite{Stan1}. We may assume that $n\ge1$.
It is enough to prove \eqref{Stirlingfalling} for positive integers $x$, in which case  $x^n$ is the number of functions $f:[n]\to A$, where $|A|=x$.
We now count such functions in another way. For each $j\in[x]$, consider all those functions whose range contains exactly $j$ elements.
The inverse images of these elements give a partition of $[n]$ into $j$ nonempty sets, $\{B_i\}_{i=1}^j$. We can choose  the particular $j$ elements in $\binom xj$ ways,
and we can order the sets $\{B_i\}_{i=1}^j$ in $j!$ ways. Thus there are $S(n,j)\binom xjj!=S(n,j)(x)_j$ such functions.
Therefore the number of functions $f$ as above can also be represented as $\sum_{j=0}^n S(n,j)(x)_j$.

In the case of the Poisson distribution, the falling factorial moments, $E(X_\lambda)_j$, are very simple.
We have
\begin{equation}\label{ffmom}
E(X_\lambda)_j=\sum_{k=0}^\infty (e^{-\lambda}\frac{\lambda^k}{k!})(k)_j=e^{-\lambda}\sum_{k=0}^\infty
\frac{\lambda^k}{(k-j)!}=\lambda^j,\ j\ge0.
\end{equation}
Substituting $X_\lambda$ for $x$ in \eqref{Stirlingfalling}, taking expectations on both sides of the equation and using
\eqref{ffmom} gives
\begin{equation}\label{Touchardmoment}
EX_\lambda^n=\sum_{j=0}^nS(n,j)E(X_\lambda)_j=\sum_{j=0}^nS(n,j)\lambda^j.
\end{equation}
The \it Touchard polynomials\rm\ (see, for example, \cite{Man2}), are defined
by
\begin{equation}\label{Touchard}
T_n(x)=\sum_{j=0}^nS(n,j)x^j,\ \ n=0,1,\cdots.
\end{equation}
(For $n\ge1$, we can begin the summation above from $j=1$.)
Thus, from \eqref{Touchardmoment}, the moments of a Poisson--distributed random variable can be written
 in terms of the Touchard polynomials as
 \begin{equation}\label{Touchardmomentagain}
EX_\lambda^n=T_n(\lambda),\ n=0,1,\cdots.
 \end{equation}

Recall that for each $n\in\mathbb{N}$,  the \it Bell number \rm  $B_n$  counts the total number of partitions of  the set  $[n]$;
that is,
\begin{equation}\label{Bell}
B_n=\sum_{k=0}^nS(n,k),\ n=1,2,\cdots.
\end{equation}
It then follows from \eqref{Touchardmomentagain} that
\begin{equation}\label{BellPoiss}
EX_1^n=B_n,\ n=1,2,\cdots.
\end{equation}
Thus, the $n$th moment of the Poisson distribution with parameter $\lambda=1$ is equal to
the total number of partitions of the set $[n]$.

In each of the   sections \ref{Dobalg} and \ref{Cyclecounts}, we give an application of  the connection between the Poisson distribution and Stirling numbers of the second kind
as expressed in \eqref{Touchardmomentagain}.

The calculations in \eqref{Stirlingfalling}-\eqref{Touchardmoment} hardly   lend   intuition
as to why the Poisson distribution should be  connected to the Stirling numbers of the second kind via
\eqref{Touchardmomentagain}.
The intuition (as well as an alternate proof of \eqref{Touchardmomentagain}) comes from the Poisson distribution's connection to combinatorics; namely, through
the Poisson approximation to the binomial distribution.

\noindent \it Alternate proof of \eqref{Touchardmomentagain}.\rm\ We may assume that $n\ge1$.
The $n$th moment of a random variable distributed as Bin$(N,\frac\lambda N)$ is given by
\begin{equation}\label{momentbin}
EX^n_{N,\frac\lambda N}=\sum_{j=0}^Nj^n\binom Nj(\frac\lambda N)^j(1-\frac\lambda N)^{N-j}.
\end{equation}
Since
$$
\sum_{j=M}^Nj^n\binom Nj(\frac\lambda N)^j(1-\frac\lambda N)^{N-j}\le \sum_{j=M}^\infty\frac{j^n}{j!}\lambda^j, \
0\le M\le N,
$$
we have $\lim_{M\to\infty}\sup_{N\ge M}\sum_{j=M}^Nj^n\binom Nj(\frac\lambda N)^j(1-\frac\lambda N)^{N-j}=0$.
This in conjunction with \eqref{PAB} shows that
\begin{equation}\label{PABmom}
EX_\lambda^n=\lim_{N\to\infty} EX_{N,\frac\lambda N}^n, \ n=0,1,\cdots.
\end{equation}

Writing  $X_{N,\frac\lambda N}$
as a sum of indicator random variables,
$$
X_{N,\frac\lambda N}=\sum_{j=1}^N I^{(N)}_j,
$$
where the $\{I^{(N)}_j\}_{j=1}^N$ are IID and distributed according to  the  Bernoulli distribution with parameter
$\frac\lambda N$,
we have
\begin{equation}\label{intuitive1}
X_{N,\frac\lambda N}^n=\sum_{j_1,j_2,\cdots, j_n=1}^N\prod_{i=1}^n I^{(N)}_{j_i},\ n\in\mathbb{N}.
\end{equation}
Note that
\begin{equation}\label{intuitive2}
E\prod_{i=1}^n I^{(N)}_{j_i}=(\frac\lambda N)^k,\ \text{where}\ k=|\{j_1,j_2,\cdots, j_n\}|.
\end{equation}
From here on, we  assume  that $N\ge n$. (Recall from \eqref{PABmom} that we will be letting $N\to\infty$
with $n$ fixed.)
How many $n$-tuples $(j_1,\cdots,j_n)\in[N]^n$ satisfy
$|\{j_1,j_2,\cdots, j_n\}|=k$?
We first need to choose $k$ integers from $[N]$; there are $\binom Nk$ ways to make this choice.
Having fixed $k$ such integers, we then need to set $j_i$ equal to one of these
integers, for each $i=1,\cdots, n$, and we need to do it in such a way as to ensure
that $|\{j_1,\cdots, j_n\}|=k$.
The number of ways to do this is $k!$ times the number of ways to partition a set of size $n$ into $k$ non-empty sets; namely,
$k!S(n,k)$ ways.
Using this observation along with \eqref{intuitive1} and \eqref{intuitive2}, we have
\begin{equation}\label{intuitive3}
EX_{N,\frac\lambda N}^n=\sum_{k=1}^n\binom Nkk!S(n,k)(\frac\lambda N)^k,\ n\in\mathbb{N}.
\end{equation}
From \eqref{intuitive3}, we conclude that
\begin{equation}\label{intuition4}
\lim_{N\to\infty}EX_{N,\frac\lambda N}^n=\sum_{k=0}^nS(n,k)\lambda^k=T_n(\lambda),\ n=0,1,\cdots.
\end{equation}
From \eqref{intuition4} and \eqref{PABmom}, we obtain
\eqref{Touchardmomentagain}.
\hfill $\square$

\section{Dobi\'nski's formula for Bell numbers and an algorithm for selecting a random partition of $[n]$}
\label{Dobalg}
Considering \eqref{nthmom} along with
\eqref{BellPoiss} yields the beautiful formula
\begin{equation}\label{Dob}
B_n=e^{-1}\sum_{k=0}^\infty \frac{k^n}{k!},\ n\in\mathbb{N}.
\end{equation}
It is quite striking that the sequence $\{B_n\}_{n=1}^\infty$ of combinatorial numbers can  be
represented as such an infinite series with a parameter.
This result is known as \it Dobi\'nski's formula \cite{Dob,Man1, Man2,Pitman97, Rota}.\rm\
As an interesting application of Dobi\'nski's formula, we define  a constructive algorithm that chooses
uniformly at random a partition of $[n]$; that is, it  chooses at random one of the $B_n$ partitions of $[n]$ in such a way that each partition has probability $\frac1{B_n}$ of being chosen \cite{Stam,Pitman97}.

In order to choose a partition of $[n]$, we must choose a $k\in[n]$ that determines  how many sets will be in the partition,
and then given $k$, we must choose one of the $S(n,k)$ partitions of $[n]$ into $k$ non-empty subsets.
We  utilize a balls-in-bins type of construction.
For each $n$, let $M_n$ be a positive, integer-valued random variable with distribution
\begin{equation}\label{Mndist}
P(M_n=m)=\frac1{eB_n}\frac{m^n}{m!},\ m=1,2,\cdots.
\end{equation}
Note that this is indeed a probability distribution in light of Dobi\'nski's formula.
Now take $M_n$ bins, and sequentially place $n$ balls, numbered from 1 to $n$, into these bins, uniformly at random.
Let $K_n$ denote the number of non-empty bins. (Of course, $K_n$ is random.)
 We have now constructed a random set partition of $[n]$ into $K_n$ non-empty sets.
Denote it by $\Psi_n$.
\medskip

\noindent \it Proof that the random partition $\Psi_n$ is uniformly distributed.\rm\
We need to show that $P(\Psi_n=\psi)=\frac1{B_n}$, for every partition $\psi$  of $[n]$.
Fix a partition $\psi$ of $[n]$.
Let $k$ denote the number of sets in the partition $\psi$.
%:\ $k:=|\lambda|$. \pause
If $M_n=m<k$, then it is clear from the construction that   it is not possible to have  $\Psi_n=\psi$.
That is, the conditional probability that $\Psi_n=\psi$ given that $M_n=m$ is equal to zero:
\begin{equation}\label{cond1}
P(\Psi_n=\psi|M_n=m)=0,\ \text{if}\ m<k.
\end{equation}
On the other hand,  for $m\ge k$, we claim that
\begin{equation}\label{cond2}
P(\Psi_n=\psi|M_n=m)=\frac{(m)_k}{m^n}=\frac{m(m-1)\cdots(m-k+1)}{m^n},\  m\ge k.
\end{equation}
To see this, note first that conditioned on  $M_n=m$, there are $m^n$ possible ways to place the  balls into the  bins.
Thus, to  prove \eqref{cond2}, we need to  show that of these $m^n$ ways, exactly  $m(m-1)\cdots(m-k+1)$ of them  result in constructing the partition $\psi$. In the paragraph after the next one, we will prove this fact. In order to make the proof more reader
friendly,  in the next paragraph, we  consider what must happen to the first three balls  in two specific examples.

Let $3\le k\le n$ and of course, as in \eqref{cond2}, let  $m\ge k$. Assume first, for example, that the numbers 1 and 3 appear
in the same subset of the partition $\psi$, and that the number 2 appears in a different subset of $\psi$.
Then in order to construct the partition $\psi$, there are $m$ bins to choose from for ball number 1, there are
$m-1$ bins to choose from for ball number 2, namely any bin except for the one containing ball number 1, and there is one bin to choose from for ball number 3, namely the bin
containing ball number 1.
Now assume alternatively that the numbers 1,2 and 3 all appear in different subsets in the partition $\psi$.
Then in order to construct the partition $\psi$, there are $m$ bins to choose from for ball number 1, there are
$m-1$ bins to choose from for ball number 2, and there are $m-2$ bins to choose from for ball number 3.

We now give the full proof.
Let $l_1=0$ and for $j=2,\cdots, n$, let $l_j$ equal the number of subsets in the partition
 $\psi$ that contain at least one of the numbers in $[j-1]$.  For each $j\in[n]$, consider the situation
 where balls numbered 1 through $j-1$ have been placed in bins in such a way that
 it is still possible to end up with the partition $\psi$, and then let
  $\gamma_j$ denote the number of bins
in which   ball number $j$ can be placed in    order to continue to  preserve
the possibility of ending up with the partition $\psi$.
Of course $\gamma_1=m$. For $j\ge2$, $\gamma_j$ is equal to 1 if in  $\psi$, the subset containing $j$ also contains an element from $[j-1]$, while otherwise, $\gamma_j=m-l_j$.
Furthermore, for $j\ge1$, if $\gamma_j=1$, then $l_{j+1}=l_j$, while if $\gamma_j=m-l_j$, then
$l_{j+1}=l_j+1$.
From this analysis, it follows that the number of ways to place the $n$ balls into the $m$ bins in such a way as to
result in constructing $\psi$ is $\prod_{j=1}^n\gamma_j=m(m-1)\cdots(m-j+1)$.

From \eqref{Mndist}-\eqref{cond2} we conclude that
$$
\begin{aligned}
&P(\Psi_n=\psi)=\sum_{m=1}^\infty P(M_n=m)P(\Psi_n=\psi|M_n=m)=\\
&\sum_{m=k}^\infty \frac1{eB_n}\frac{m^n}{m!}\frac{(m)_k}{m^n}= \frac1{eB_n} \sum_{m=k}^\infty \frac1{(m-k)!}=\frac1{B_n},
\end{aligned}
$$
which completes the proof that the algorithm yields a uniformly random partition of $[n]$.
\hfill $\square$

We note that in \cite{Stam} it is also proved that the number of empty bins
appearing in the construction is independent of the particular partition produced, and is distributed
according to the Poisson distribution with parameter 1.

For some more on probabilistic aspects of set partitions, see \cite{Man1,Pitman97}.
For some results concerning the asymptotic behavior of  certain partition statistics under the uniform distribution
on partitions of $[n]$, see \cite{Harper,Sach}.

\section{Chinese restaurant and Feller constructions of random permutations with applications}
\label{Chinese}
Consider the space $S_n$ of permutations of $[n]$ with the uniform probability measure
$P_n$ which gives probability $\frac1{n!}$ to each permutation $\sigma\in S_n$.
We present the so-called \it Chinese restaurant construction\rm\ which simultaneously yields a uniformly random permutation
$\Sigma_n$ in $S_n$, for all $n$ \cite{Pitman06}. Furthermore, this construction is consistent  in the sense that if one writes out
the permutation $\Sigma_n$ as the product of its cycles  and deletes the number $n$ from the cycle in which it appears, then
the resulting random permutation of $S_{n-1}$ is equal to $\Sigma_{n-1}$.

The construction works as follows. Consider a restaurant with an unlimited number of circular tables, each of which
has an unlimited number of seats. Person number 1 sits at a table.  Now for $n\ge1$, suppose that   persons
number 1 through $n$ have already been seated. Then person number $n+1$ chooses a seat as follows.
For each $j\in[n]$, with probability $\frac1{n+1}$, person number $n+1$ chooses to sit to the left of person number $j$.
Also, with probability $\frac1{n+1}$, person number $n+1$ chooses to sit at an unoccupied table.
Now for each $n\in\mathbb{N}$, the random permutation $\Sigma_n\in S_n$ is defined by $\Sigma_n(i)=j$, if
after the first $n$ persons have taken seats,
person number $j$ is seated to the left of person number $i$.

To see that $\Sigma_n$ is a uniformly random permutation in $S_n$, proceed by induction.
It is  true for $n=1$. Now assume it is true for some $n\ge1$.
Let $\sigma\in S_{n+1}$, and let $\sigma'\in S_n$ be the permutation obtained by writing
$\sigma$ as a product of cycles and deleting $n+1$. By  the inductive assumption,
$P(\Sigma_n=\sigma')=\frac1{n!}$. By the construction, $P(\Sigma_{n+1}=\sigma|\Sigma_n=\sigma')=\frac1{n+1}$
and $P(\Sigma_{n+1}=\sigma|\Sigma_n\neq\sigma')=0$.
Thus,
$$
\begin{aligned}
&P(\Sigma_{n+1}=\sigma)=P(\Sigma_n=\sigma')P(\Sigma_{n+1}=\sigma|\Sigma_n=\sigma')+\\
&P(\Sigma_n\neq\sigma')P(\Sigma_{n+1}=\sigma|\Sigma_n\neq\sigma')=
\frac1{n!}\frac1{n+1}=\frac1{(n+1)!}.
\end{aligned}
$$
The above-noted consistency is also clear.

The  construction above allows for easy analysis of certain properties of random permutations.
We illustrate this with regard to the number of cycles in a random permutation.
Let $N^{(n)}(\sigma)$ denote the number of cycles in the permutation $\sigma\in S_n$.
Then $N^{(n)}$ is  a random variable on the  probability space $(S_n,P_n)$.
Fix $n\in\mathbb{N}$. The number of cycles in the random permutation $\Sigma_n$ is the number of persons
from among persons 1 through $n$ who chose to sit at an unoccupied  table.
The probability that person $j$ chose to sit at an unoccupied table is $\frac1j$. It is clear that
whether or not person $j$ chose to sit at an empty table is independent of whether or not any of the other
persons chose to do so. In light of this, it follows that the distribution of the random variable $N_n$
is equal to the distribution of
$\sum_{i=1}^n I_i$,
where $\{I_i\}_{i=1}^n$ are independent and for each $i$, $I_i$ has the  Bernoulli distribution
with parameter $\frac1i$.
From this it is immediate that
$$
EN^{(n)}=\sum_{i=1}^n\frac1i\sim\log n\ \text{as}\ n\to\infty.
$$
Recalling that the variance of a sum of independent random variables is equal to the sum of their variances,
a simple calculation gives
$$
\text{Var}(N^{(n)})=\sum_{i=1}^n\frac{i-1}{i^2}\sim\log n\  \text{as}\ n\to\infty.
$$
Now the second moment method \eqref{secondmommeth}
yields the weak  law of large numbers for $N^{(n)}$.
\begin{theorem}\label{N^n1}
The total number of cycles $N^{(n)}$ in a uniformly random permutation from $S_n$ satisfies
\begin{equation}
\frac{N^{(n)}}{\log n}\stackrel{\text{dist}}{\to}1.
\end{equation}
\end{theorem}
And
some rather basic probabilistic machinery yields
the central limit theorem for $N^{(n)}$ \cite{Durrett}.

\begin{theorem}
The total number of cycles $N^{(n)}$ in a uniformly random permutation from $S_n$ satisfies
\begin{equation}
\frac{N^{(n)}-\log n}{\sqrt{\log n}}\stackrel{\text{dist}}{\to}N(0,1).
\end{equation}
\end{theorem}

We now present an  alternative construction of a random permutation in $S_n$.
The construction,
called the Feller construction \cite{Feller0,ABT}, builds the  permutation's cycles one by one.
 Begin with the number 1. Choose uniformly at random a number $j$ from $[n]$ and set it equal to $\sigma_1$:
 $\sigma_1=j$. If $j=1$, then we have constructed a complete cycle (of length one) in the permutation. If $j\neq1$,
 choose uniformly at random a number $k$ from $[n]-\{j\}$ and set it equal to $\sigma_j$: $\sigma_j=k$.
If $k=1$, then we have constructed a complete cycle (of length two). If $k\neq1$,  choose uniformly at random
a number $l$ from $[n]-\{j,k\}$ and set it equal to $\sigma_k$: $\sigma_k=l$.
If $l=1$, then we have constructed a complete cycle (of length three). If $l\neq1$,  choose uniformly at random
a number $m$ from $[n]-\{j,k,l\}$ and set it equal to $\sigma_l$: $\sigma_l=m$. Continue like this until the number 1 is finally chosen and a cycle is completed. Once a cycle is completed, start the process over again beginning, say, with the smallest number that has not been used in the completed cycle.
This number now takes on the role that the number 1 had above.
After $n$ selections of a number uniformly at random, each time  from a set whose size
has shrunk by one from the previous selection, the process ends and the construction of the permutation is completed.
Denote the permutation constructed here by $\Sigma'_n$. It is clear from the construction that $\Sigma'_n$ is uniformly
distributed on $S_n$.
As an example, let $n=6$. If we first select  uniformly at random from $[6]$ the number 4, and then select  uniformly at random from $[6]-\{4\}$ the number 1, then we have $\sigma_1=4$ and
$\sigma_4=1$, completing a cycle. Now we start with the smallest number not yet used, which is 2. If we select uniformly at random from $[6]-\{1,4\}$ the number 2, then $\sigma_2=2$, and we have completed another cycle. Now we start again
with the smallest number not yet used, which is 3. If we select uniformly at random from $[6]-\{1,2,4\}$ the number
6, then we have $\sigma_3=6$. If then we select uniformly at random from $[6]-\{1,2,4,6\}$ the number 5, then
$\sigma_6=5$. The last step of course is deterministic; we must choose $\sigma_5=3$. We have constructed the permutation
$\Sigma'_6=426135$ (where the permutation has been written in one-line notation).

From the above description, it is clear that   the probability of completing a cycle at the $j$th selection of a number is
$\frac1{n-j+1}$; indeed, at the $j$th selection, there are $n-j+1$ numbers to choose from, and only one of them completes
a cycle. Define
the indicator random variable $I_{n-j+1}$ to be one if a cycle was completed at the $j$th selection, and  zero otherwise.
Then
$I_{n-j+1}$ has the following Bernoulli distribution
with parameter $\frac1{n-j+1}$:
$P(I_{n-j+1}=1)=1-P(I_{n-j+1}=0)=\frac1{n-j+1}$. A little thought also shows that the $\{I_i\}_{i=1}^n$ are mutually independent;
knowledge of the values of certain of the $I_i$'s has no influence on the probabilities concerning the other
$I_i$'s.
Thus, as with the Chinese restaurant construction, the Feller construction leads immediately to the fact
that $N_n$ is distributed as $\sum_{i=1}^nI_i$.

For $j\in[n]$ and $\sigma\in S_n$,  let $L^{(n);j}(\sigma)$ denote the length of the  cycle  in
$\sigma$ that contains $j$.
Then $L^{(n);j}$ is a random variable on the probability space $(S_n,P_n)$.
Since $P_n$ is the uniform distribution, it is clear that the distribution
of $L^{(n);j}$  is the same for all $j\in[n]$.
Using the Feller construction, we easily obtain the following result.
\begin{proposition}\label{cyclesizeunif}
Under $P_n$, for each $j\in[n]$, the random variable $L^{(n);j}$ is uniformly distributed
on $[n]$.
\end{proposition}
\begin{proof}
As noted above, it suffices to consider $j=1$. From the Feller construction, the probability
that for the uniformly random permutation  $\Sigma'_n$ in $S_n$, the cycle  containing 1 is of size $j$ is given by
$\frac1n$, for $j=1$, and is given by
$\frac{n-1}n\frac{n-2}{n-1}\cdots \frac {n-j+1}{n-j+2}\frac1{n-j+1}=\frac1n$, for $j=2,\cdots, n$.
\end{proof}

Now for each $n\in\mathbb{N}$ and $j\in[n]$,  let $A^{(n)}_j=A^{(n)}_j(\Sigma'_n)$ denote the length of the $j$th
cycle constructed via the Feller construction, with $A^{(n)}_j=0$ if fewer than $j$ cycles were constructed.
We  refer to $\{A^{(n)}_j\}_{j=1}^n$ as the \it ordered cycles.\rm\
Note that whereas $L^{(n);j}$ was defined on $(S_n,P_n)$, independent of any particular construction,
$A^{(n)}_j$ is defined in terms of the Feller construction.
Let $\{U_n\}_{n=1}^\infty$ be a sequence of IID random variables with the uniform
distribution on $[0,1]$. Let $X_1=U_1$ and $X_n=(1-U_1)\cdots(1-U_{n-1})U_n$, for $n\ge2$.
The $\{X_n\}_{n=1}^\infty$ can be understood in terms of  the uniform \it stick breaking\rm\ model: take a stick of length one and break it at a uniformly random point.
Let the length of the left hand piece be $X_1$. Take the remaining piece of stick, of length $1-X_1$, and break it at a uniformly random point.
Let the length of the left hand piece be $X_2$, etc.
In light of Proposition \ref{cyclesizeunif} and the Feller construction, the following theorem is almost immediate.
\begin{theorem}\label{orderedcycles}
For any $k\in\mathbb{N}$,
 the random vector\newline $\frac1n(A^{(n)}_1,\cdots, A^{(n)}_k)$
converges in distribution to the random vector $(X_1,\cdots, X_k)$, where
$$
X_j=(1-U_1)\cdots(1-U_{j-1})U_j,\ j\in \mathbb{N},
$$
and
$\{U_n\}_{n=1}^\infty$ are IID random variables distributed according to the uniform distribution on
$[0,1]$.
\end{theorem}
The distribution of $\{X_n\}_{n=1}^\infty$ is known as the GEM distribution \cite{ABT}.
It will be mentioned  again in section \ref{DickmanBuch}.

\section{Convergence in distribution of
  cycle counts of fixed length in random permutations to Poisson distributions}\label{Cyclecounts}
In this section we present a
second application of the connection between the Poisson distribution and Stirling numbers of the second kind
as expressed in \eqref{Touchardmomentagain}.
As in the previous section, we consider the space $S_n$ of permutations of $[n]$ with the uniform probability measure
$P_n$.
For each $m\in[n]$, let $C^{(n)}_m:S_n\to\mathbb{N}$ count the number of
$m$-cycles in a permutation; that is, $C^{(n)}_m(\sigma)$ equals the number of $m$-cycles in $\sigma\in S_n$.
Then
 $C_m^{(n)}$ is a random variable on the probability space $(S_n,P_n)$.
We use the standard probability  notation $P_n(C_m^{(n)}=k)$:
$$
P_n(C_m^{(n)}=k)=P_n(\{\sigma\in S_n: C_m^{(n)}(\sigma)=k\})=\frac{|\{\sigma\in S_n: C_m^{(n)}(\sigma)=k\}|}{n!}.
$$
For fixed $n$, the  distribution of $C_m^{(n)}$ is complicated, but we will prove that as $n\to\infty$, the distribution
of $C_m^{(n)}$ converges in distribution to the  Poisson$(\frac1m)$ distribution.
That is,
\begin{equation}\label{cyclespoisson}
\lim_{n\to\infty}P_n(C_m^{(n)}=k)=e^{-\frac1m}\frac{(\frac1m)^k}{k!},\ k=0,1,\cdots;\ m=1,2,\cdots.
\end{equation}
In fact, a  stronger result holds; namely that for any $m\in\mathbb{N}$, the distribution of the random vector $(C_1^{(n)},\cdots, C_m^{(n)})$ converges in distribution to a random vector, call it
$(X_1,\cdots, X_m)$, with independent components, where $X_i$ is distributed according to the Poisson distribution with parameter $\frac1i$.
In the theorem below,
the product on the right hand side  expresses the independence of the components $\{X_i\}_{i=1}^m$.
\begin{theorem}\label{cyclespoissonmulti}
Let $C_i^{(n)}$ denote the total number of cycles of length $i$ in  a uniformly random permutation from $S_n$.
Then
\begin{equation}\label{cyclespoissonindep}
\begin{aligned}
&\lim_{n\to\infty}P_n(C_1^{(n)}=k_1,\cdots, C_m^{(n)}=k_m)=\prod_{i=1}^me^{-\frac1{i}}\frac{(\frac1i)^{k_i}}{k_i!},\\
&\text{for}\  k_i=0,1,\cdots,\ 1\le i\le m;\ m=1,2,\cdots.
\end{aligned}
\end{equation}
\end{theorem}

The proof we give of \eqref{cyclespoisson} readily generalizes to a proof Theorem \ref{cyclespoissonmulti}, but the notation
is more cumbersome.
The combinatorial part of our proof is contained in the following proposition.
\begin{proposition}\label{likepoissonmoments}
For $n\ge mk$,
\begin{equation*}\label{finaltouch}
\begin{aligned}
&E_n(C^{(n)}_m)^k=T_k(\frac1m).
\end{aligned}
\end{equation*}
\end{proposition}
The proof of \eqref{cyclespoisson} follows from Proposition \ref{likepoissonmoments},
\eqref{Touchardmomentagain} and the method of moments described in the final paragraph of section \ref{tools}.
To see that the condition  required to employ the method of moments indeed holds  in the case
at hand,  see, for example, the explanation in the first two paragraphs of the proof of Theorem C in \cite{Pin17}.
\medskip

We note that there are  ways other than the method of moments  to prove \eqref{cyclespoisson} and Theorem \ref{cyclespoissonmulti}. See for example
\cite{ABT} which uses \eqref{cauchy} in Chapter \ref{DickmanBuch} and the method of inclusion-exclusion, or \cite{NP} for a completely different approach.
A proof  of  \eqref{cyclespoissonindep} in \cite{Wilf} (and also in \cite{Pin14}) uses generating functions in a rather involved way.
The first proof of  \eqref{cyclespoisson} seems to be in \cite{Gon}  and the proof
of \eqref{cyclespoissonindep} may go back to \cite{Kolc}.
\medskip

\noindent \it Proof of Proposition \ref{likepoissonmoments}.\rm\
Assume that $n\ge mk$.
For $D\subset[n]$ with $|D|=m$,
let $1_D(\sigma)$ be the indicator random variable equal to 1 or  0 according to whether or not $\sigma\in S_n$ possesses an $m$-cycle consisting of
the elements of $D$. Then we have
\begin{equation}\label{Cnrep}
C^{(n)}_m(\sigma)=\sum_{\stackrel{D\subset[n]}{|D|=m}}1_D(\sigma),
\end{equation}
and
\begin{equation}\label{Cnmkthmoment}
E_n(C^{(n)}_m)^k=\sum_{\{(D_1,\cdots, D_k)\subset[n]^k:|D_j|=m,j\in[k]\}}E_n\prod_{j=1}^k1_{D_j}.
\end{equation}

Now $E_n\prod_{j=1}^k1_{D_j}\neq0$ if and only if for some $l\in[k]$, there exist
disjoint sets $\{A_i\}_{i=1}^l$ such that $\{D_j\}_{j=1}^k=\{A_i\}_{i=1}^l$.
If this is the case, then
\begin{equation}\label{ldistinct}
E_n\prod_{j=1}^k1_{D_j}=E_n\prod_{j=1}^l1_{A_j}=\frac{(n-lm)!((m-1)!)^l}{n!}.
\end{equation}
(Here we have used the assumption that $n\ge mk$.)
The number of ways to construct a collection of  $l$ disjoint sets  $\{A_i\}_{i=1}^l$, each of which consists of $m$ elements from $[n]$,
is
$\frac{n!}{(m!)^l(n-lm)!\thinspace l!}$.
Given the $\{A_i\}_{i=1}^l$, the number of ways to choose the sets $\{D_j\}_{j=1}^k$ so that $\{D_j\}_{j=1}^k=\{A_i\}_{i=1}^l$ is
equal to the Stirling number of the second kind $S(k,l)$, the number of ways to partition a set of size $k$ into
$l$ nonempty parts, multiplied by $l!$, since the labeling must be taken into account.
From these facts along with  \eqref{Cnmkthmoment} and \eqref{ldistinct},
we conclude that for $n\ge mk$,
\begin{equation*}\label{finaltouch}
\begin{aligned}
&E_n(C^{(n)}_m)^k=\\
&\sum_{l=1}^k\big(\frac{(n-lm)!((m-1)!)^l}{n!}\big)\big(\frac{n!}{(m!)^l(n-lm)!}\big)S(k,l)=\\
&\sum_{l=1}^k\frac1{m^l}S(k,l)=T_k(\frac1m).
\end{aligned}
\end{equation*}
\hfill $\square$

\section{Limiting behavior of the lengths of the largest and smallest cycles in random permutations and the connection to the  Dickman and Buchstab functions}\label{DickmanBuch}

Theorem \ref{cyclespoissonmulti} deals with the limiting distributions of the number of cycles of \it fixed\rm\ lengths. We now consider
 the limiting behavior of the \it lengths\rm\ of the  largest cycles or of the smallest cycles. We begin by considering the largest cycles.

One way to analyze this begins with the formula of Cauchy which counts the number of permutations
with given cycle numbers. Let $\{c_j\}_{j=1}^n$ be nonnegative integers. Cauchy's formula \cite{ABT} states that the number of permutations $\sigma\in S_n$ that
satisfy $C^{(n)}_j(\sigma)=c_j$, for all $j\in[n]$,
is equal to
$n!\prod_{j=1}^n(\frac1j)^{c_j}\frac1{c_j!}$, if $\sum_{j=1}^njc_j=n$, and of course is equal to zero otherwise. Therefore, for a random permutation of $S_n$, we have
\begin{equation}\label{cauchy}
P_n(C^{(n)}_1=c_1,\cdots, C^{(n)}_n=c_n)=
%1_{\{\sum_{j=1}^njc_j=n\}}
\prod_{j=1}^n(\frac1j)^{c_j}\frac1{c_j!}, \ \text{if}\ \sum_{j=1}^njc_j=n.
\end{equation}
The product on the right hand side above reminds one of  Poisson distributions.
Let $\{Z_j\}_{j=1}^\infty$ be a sequence of independent  random variables, with $Z_j\sim\text{Poisson}(\frac1j)$, and define
$T_{[n]}=\sum_{j=1}^njZ_j$, for $n\in\mathbb{N}$.
We have
\begin{equation}\label{cauchyagain}
\begin{aligned}
&P(Z_1=c_1,\cdots, Z_n=c_n,T_{[n]}=n)=\prod_{j=1}^ne^{-\frac1j}(\frac1j)^{c_j}\frac1{c_j!}=\\
&e^{-\sum_{j=1}^n\frac1j}\prod_{j=1}^n(\frac1j)^{c_j}\frac1{c_j!},
\  \text{if}\ \sum_{j=1}^njc_j=n,
\end{aligned}
\end{equation}
since the requirement $T_{[n]}=n$ is automatically fulfilled if $\sum_{j=1}^njc_j=n$.
Summing \eqref{cauchy} over all $\{c_j\}_{j=1}^n$ which satisfy $\sum_{j=1}^njc_j=n$ gives
\begin{equation}\label{normalization}
\sum_{\substack{(c_1,\cdots, c_n):\\\sum_{j=1}^njc_j=n}}\prod_{j=1}^n(\frac1j)^{c_j}\frac1{c_j!}=1.
\end{equation}
Summing \eqref{cauchyagain} over this same set of $\{c_j\}_{j=1}^n$, and using \eqref{normalization} gives
\begin{equation}\label{T0n}
P(T_{[n]}=n)=e^{-\sum_{j=1}^n\frac1j}.
\end{equation}
From \eqref{cauchy}, \eqref{cauchyagain} and \eqref{T0n}, we conclude that
\begin{equation}\label{conditionrelation}
P_n(C^{(n)}_1=c_1,\cdots, C^{(n)}_n=c_n)=P(Z_1=c_1,\cdots, Z_n=c_n|T_{[n]}=n).
\end{equation}

The representation in \eqref{conditionrelation} of the distribution of the lengths of the cycles in terms of a conditioned distribution of independent random variables can be exploited.
We demonstrate this by sketching the method in the case of the longest cycle \cite{ABT}. Let $L_j^{(n)}(\sigma)$ denote the length of the $j$th longest cycle in $\sigma\in S_n$, $j=1,2,\cdots$.
We will consider only $L_1^{(n)}$, but one can work similarly with the random vector $(L_1^{(n)},\cdots, L_j^{(n)})$.
%Similar methods can also be used to treat the shortest cycles.
The event $\{L_1^{(n)}\le m\}$ can be written as $\{C_{m+1}^{(n)}=C_{m+2}^{(n)}=\cdots= C_n^{(n)}=0\}$.
Using \eqref{conditionrelation} this gives
\begin{equation}\label{exploitindep}
\begin{aligned}
&P_n(L^{(n)}_1\le m)=P(Z_{m+1}=\cdots =Z_n=0|T_{[n]}=n)=\\
&\frac{P(Z_{m+1}=\cdots =Z_n=0, T_{[n]}=n)}{P(T_{[n]}=n)}=\frac{P(Z_{m+1}=\cdots =Z_n=0, T_{[m]}=n)}{P(T_{[n]}=n)}=\\
&\frac{P(Z_{m+1}=\cdots =Z_n=0)P(T_{[m]}=n)}{P(T_{[n]}=n)}=\frac{P(T_{[m]}=n)}{P(T_{[n]}=n)}\exp(-\sum_{j=m+1}^n\frac1j),
\end{aligned}
\end{equation}
where we have used the independence of $T_{[m]}$ and $\{Z_j\}_{j=m+1}^n$ in the next to the last equality.

It is known that the random variable
$\frac1nT_{[n]}=\frac1n\sum_{j=1}^n jZ_j$   converges in distribution as $n\to\infty$ to the so-called \it Dickman distribution\rm\ \cite{ABT,Pin18}.
This distribution, supported on $[0,\infty)$, has density
$e^{-\gamma}\rho(x)$, where $\gamma$ is Euler's constant and
$\rho$ is the \it Dickman function\rm, defined as   the unique continuous function satisfying the differential-delay equation
\begin{equation}\label{Dickmanequ}
\begin{aligned}
&\rho(x)=1,\ x\in(0,1];\\
&x\rho'(x)+\rho(x-1)=0, \ x>1.
\end{aligned}
\end{equation}
(The proof that $\int_0^\infty\rho(x)dx=e^\gamma$ is not immediate \cite{MV}, \cite{Tene}.)
For use below, we note that
\begin{equation}\label{integralident}
\int_{x-1}^x\rho(y)dy=x\rho(x), \ x\ge1.
\end{equation}
This identity  follows by defining $H(x)=\int_{x-1}^x\rho(y)dy$ and using
\eqref{Dickmanequ} to obtain $H'(x)=(x\rho(x))'$ and $H(1)=1$.
For more on the interesting topic of  convergence in distribution to the Dickman distribution, see \cite{Pin18} and references therein.
In \eqref{exploitindep}, replace $m$ by $m_n$ and assume that $m_n\sim nx$ as $n\to\infty$, where $x\in(0,1)$.
Then the quotient on the right hand side of \eqref{exploitindep} can be written as
\begin{equation}\label{quotientprob}
\frac{P(T_{[m_n]}=n)}{P(T_{[n]}=n)}=\frac{P(\frac1{m_n}T_{[m_n]}=\frac n{m_n})}{P(\frac1nT_{[n]}=1)}.
\end{equation}
On the right hand side above, both $\frac1{m_n}T_{[m_n]}$ and $\frac1nT_{[n]}$ converge in distribution to the Dickman distribution, and $\lim_{n\to\infty}\frac n{m_n}=\frac1x$.
Convergence in distribution to the Dickman distribution yields
\begin{equation}\label{weakdickman}
\lim_{N\to\infty}P(\frac1NT_{[N]}\in [a,b])=
\int_a^be^{-\gamma}\rho(x)dx, \ 0\le a<b<\infty,
\end{equation}
 but it gives no information  about the point probabilities
on the right hand side of \eqref{quotientprob}.
However, applying a technique called \it size-biasing\rm, which has particularly nice properties in the case of sums of independent Poisson random variables, one can express the point probabilities of $T_{[N]}$ in terms of interval probabilities:
\begin{equation}\label{pointinterval}
P(T_{[N]}=k)=\frac1k\sum_{l=1}^NP(T_{[N]}=k-l),\ k=1,2,\cdots.
\end{equation}
Using  \eqref{weakdickman} and \eqref{pointinterval} with $N=n$ and with $N=m_n$,  along with \eqref{quotientprob} and \eqref{integralident}, we obtain
\begin{equation}\label{almostfinaldickman}
\lim_{n\to\infty}\frac{P(T_{[m_n]}=n)}{P(T_{[n]}=n)}=\int_{\frac1x-1}^\frac1x\rho(y)dy=\frac1x\rho(\frac1x), \ \text{if}\ \lim_{n\to\infty}\frac{m_n}n=x\in(0,1).
\end{equation}
Substituting $m_m\sim nx$ in \eqref{exploitindep}, noting that
$\lim_{n\to\infty}\exp(-\sum_{j=m_n+1}^n\frac1j)=x$, and using \eqref{almostfinaldickman}, we arrive at the following result.
 \begin{theorem}\label{largestcycle}
Let $L^{(n)}_1$ denote the length of the largest cycle in a uniformly random permutation from $S_n$.
Then $\frac1nL^{(n)}_1$ converges in distribution to the distribution whose distribution function is given by
$\rho(\frac1x),\ x\in[0,1]$. That is,
$$
\lim_{n\to\infty}P_n(\frac1nL^{(n)}_1\le x)=\rho(\frac1x),\ x\in[0,1].
$$
\end{theorem}
As noted, we have followed \cite{ABT} for the proof of Theorem \ref{largestcycle}; the original proof
appears in \cite{SL}.

The distribution arising above, whose distribution function is $\rho(\frac1x)$, for $x\in[0,1]$,   is the first component of the \it Poisson-Dirichlet distribution\rm,
which can be defined as follows.
Recall the GEM distribution, that is, the  stick-breaking model $\{X_n\}_{n=1}^\infty$ introduced at the end of section \ref{Chinese}
and appearing in Theorem \ref{orderedcycles}. Let $\{\hat X_n\}_{n=1}^\infty$ denote the decreasing rearrangement of $\{X_n\}_{n=1}^\infty$.
The Poisson-Dirichlet distribution can be defined as the
 distribution of $\{\hat X_n\}_{n=1}^\infty$.
 See \cite{ABT} and references therein.
  In particular, the distribution of $\hat X_1$ is the distribution
arising in Theorem \ref{largestcycle}.
It can be shown that for any $j\in\mathbb{N}$,  the random vector $(L_1^{(n)},\cdots, L_j^{(n)})$ converges in distribution to the distribution of
$(\hat X_1,\cdots, \hat X_j)$ \cite{ABT}.

In number theory,  the Poisson-Dirichlet distribution comes up in a parallel fashion  in relation to  \it smooth\rm\ numbers, which are
integers with no large prime divisors. Let $p^+_1(k)$ denote the largest prime factor of $k\in \mathbb{N}$. It was proved by Dickman \cite{Dickman} that
$$
\lim_{n\to\infty}\frac1n|\{k\in[n]:p^+_1(k)\le n^x\}|=\rho(\frac1x),\ x\in[0,1].
$$
See \cite{MV,Tene}.
Similar to the extension  from $L_1^{(n)}$ to the vector $(L_1^{(n)},\cdots, L_j^{(n)})$,
the above result can be extended to the vector $\big(p^+_1(\cdot),\cdots, p^+_j(\cdot)\big)$,
where $p^+_i(k)$ denotes the $i$th largest prime factor (counted with multiplicities) of $k$, and $p_i^+(k)=0$ if $k$ has fewer than
$i$ prime factors \cite{Billings}.

\medskip

We now turn to the smallest cycles.
Let $S_j^{(n)}(\sigma)$ denote the length of the $j$th shortest cycle of $\sigma\in S_n$.
We have
\begin{equation*}\label{shortestones}
P_n(S_j^{(n)}>m)=P_n(\sum_{i=1}^m C_i^{(n)}<j).
\end{equation*}
Thus, by
\eqref{cyclespoissonindep},
\begin{equation}\label{shortestones2}
\lim_{n\to\infty}P_n(S_j^{(n)}>m)=P(\sum_{i=1}^mZ_i<j),
\end{equation}
where $\{Z_i\}_{i=1}^\infty$  are independent random variables and $Z_i$ has the Poission$(\frac1i)$ distribution.
A straight forward calculation shows that if $X$ and $Y$ are independent and distributed according to
Poisson$(\lambda_1)$ and Poisson$(\lambda_2)$ respectively, then
$$
P(X+Y=m)=\sum_{i=0}^mP(X=i)P(Y=m-i)=\frac{(\lambda_1+\lambda_2)^m}{m!}e^{-(\lambda_1+\lambda_2)};
$$
 that is $X+Y$ has the Poisson$(\lambda_1+\lambda_2)$ distribution.
Thus,  from \eqref{shortestones2} we have
\begin{equation}\label{shortestones3}
\lim_{n\to\infty}P_n(S_j^{(n)}>m)=\sum_{i=0}^{j-1}e^{-h(m)}\frac{(h(m))^i}{i!},\ \text{where}\ h(m)=\sum_{i=1}^m\frac1i.
\end{equation}

A  theory similar to that discussed above for the longest cycles, that also uses  \eqref{conditionrelation},
can be developed to obtain a result concerning  the so-called \it large deviations\rm\ of
$S_j^{(n)}$ \cite{ABT}.
This result  involves the \it Buchstab function\rm\ $\omega(x)$, which
is
defined for $x\ge1$ as  the unique continuous function satisfying
$$
\omega(x)=\frac1x, \ 1\le x\le 2,
$$
 and  satisfying the differential-delay equation
$$
(x\omega(x))'=\omega(x-1),\ x>2.
$$
In particular, for the shortest cycle, the result is as follows \cite{ABT}.
\begin{theorem}\label{shortestcycle}
\begin{equation*}
P_n(S_1^{(n)}>\frac nx)\sim\frac{x\omega(x)}n,\ x>1, \ \text{as}\ n\to\infty.
\end{equation*}
\end{theorem}

In number theory, the Buchstab function comes up in a parallel
  fashion  in relation to  \it rough\rm\ numbers, which are
integers with no small prime divisors. Let $p^-_1(k)$ denote the smallest prime factor of $k\in \mathbb{N}$. It was proved by
Buchstab  \cite{Buchstab} that
\begin{equation}\label{rough}
\frac1n|\{k\in[n]:p^-(k)\ge n^\frac1x\}|\sim x\omega(x)\frac1{\log n}, \ x>1, \text{as}\ n\to\infty.
\end{equation}
See \cite{MV,Tene}.
(Note that \eqref{rough} for  $x=1$ is the Prime Number Theorem.    Buchstab assumed
the Prime Number Theorem when proving his result.)

\section{The statistical behavior of random partitions of large integers}

A partition $\lambda$ of a positive integer $n$ is  a $j$-tuple  $(a_1,\cdots, a_j)$
where $j\in[n]$, $\{a_i\}_{i=1}^j$ are  positive integers satisfying $a_1\ge\cdots \ge a_j$ and $\sum_{i=1}^ja_i=n$.
Let $\Lambda_n$ denote the set of partitions of $n$.
There is a natural map, call it $\mathcal{M}_n$, from the set of permutations $S_n$ of $[n]$ onto the set of partitions of $[n]$, via the cycle decomposition of the permutation. We could write out a formal definition,
 but a simple example is probably clearer and should suffice.
Let  $\sigma\in S_9$ be given in terms of its cycle decomposition by
$\sigma=(154)(2)(386)(79)$. Then $\mathcal{M}_9(\sigma)$  is the partition (3,3,2,1) of the integer 9.
The map $\mathcal{M}_n$ along with the uniform probability measure $P_n$ on $S_n$ induces a (non-uniform) probability
measure $\mathcal{P}_n$ on $\Lambda_n$; namely,
$$
\mathcal{P}_n(\lambda)=P_n(\{\sigma:\mathcal{M}_n(\sigma)=\lambda\}),\
\text{for}\ \lambda\in\Lambda_n.
$$

%For a partition $\lambda\in\Lambda$, let
%$|\lambda|$ denote the integer of which $\lambda$ is a partition.
For $k\in\mathbb{N}$, let $X_k(\lambda)$ denote the number of parts of the  partition $\lambda\in\Lambda_n$ that are equal to $k$.
For $s\in\mathbb{N}$, let $Y_s(\lambda)$ denote the $s$-th largest part of $\lambda$, with $Y_s(\lambda)=0$
if $\lambda$ has fewer than $s$ parts.
(For convenience later on, we consider $X_k$ and $Y_s$ to be defined simultaneously on $\Lambda_n$, for all $n$.)
For example, consider the partition of 13,  $\lambda=(4,4,2,1,1,1)$. Then
$X_1(\lambda)=3$, $X_2(\lambda)=1$, $X_4(\lambda)=2$ and $X_k(\lambda)=0$ for $k\not\in\{1,2,4\}$;
$Y_1(\lambda)=Y_2(\lambda)=4$, $Y_3(\lambda)=2$, $Y_4(\lambda)=Y_5(\lambda)=Y_6(\lambda)=1$ and $Y_k(\lambda)=0$, for $k\ge 7$.

Recalling  from section \ref{Cyclecounts} that $C^{(n)}_k(\sigma)$ denotes the number of cycles of length
$k$ in the permutation $\sigma\in S_n$, it follows that
$$
X_k(\lambda)=C^{(n)}_k(\sigma), \text{where}\ \lambda=\mathcal{M}_n(\sigma).
$$
Recalling form section \ref{DickmanBuch} that $L^{(n)}_j(\sigma)$
%and $S^{(n)}_j(\sigma)$ denote respectively
denotes the length of the  $j$th longest cycle
% and the $j$th shortest cycle
in $\sigma\in S_n$, it follows
that
$$
Y_s(\lambda)=L^{(n)}_s(\sigma), \ \text{where}\ \lambda=\mathcal{M}_n(\sigma).
$$
In light of
\eqref{cyclespoisson}, it follows
that under the probability measures $\mathcal{P}_n$, the distribution of $X_k$ converges to the Poisson($\frac1k$)
distribution:
\begin{equation}\label{nonunifpart1}
\lim_{n\to\infty}\mathcal{P}_n(X_k=j)=e^{-\frac1k}\frac{(\frac1k)^j}{j!}, \ j=0,1,\cdots.
\end{equation}
And in light of Theorem \ref{largestcycle},
\begin{equation}\label{nonunifpart2}
\lim_{n\to\infty}\mathcal{P}_n(\frac{Y_1}n\le x)=\rho(\frac1x), \ x\in[0,1],
\end{equation}
where we recall that $\rho$ is the Dickman function.
We recall from the discussion after Theorem \ref{largestcycle} that  the distribution on $[0,1]$ whose distribution function
is given by $\rho(\frac1x)$ is the distribution of the first component $\hat X_1$  of the
random vector $\{\hat X_n\}_{n=1}^\infty$ whose distribution is called the
 Poisson-Dirichlet distribution.
From that discussion, it then  follows more generally that
\begin{equation}\label{nonunifpart3}
\lim_{n\to\infty}\mathcal{P}_n(\frac{Y_k}n\le x)=P(\hat X_k\le x), \ x\in[0,1], k\ge1.
\end{equation}

In this section, we consider the statistical behavior of $\Lambda_n$ under the \it uniform\rm\  measure that gives
probability $\frac1{|\Lambda_n|}$ to each partition $\lambda\in\Lambda_n$. The results are very
different from the results in \eqref{nonunifpart1}-\eqref{nonunifpart3} in  the case of the measures
$\{\mathcal{P}_n\}_{n=1}^\infty$.
 %Let $\Lambda$ denote the set
%of all partitions of  nonnegative integers, including the empty partition of 0.
%For $n=0,1,\cdots$, let $\Lambda_n$ denote the
%set of partitions of $n$.
%Let $p(n)=|\Lambda_n|$ denote the number of partitions of $n$ and let $P_n$ denote the uniform probability
%measure on $\Lambda_n$.
The paper \cite{Frist} makes beautiful use of generating functions to study the asymptotic behavior of partition
statistics in the case of the uniform distribution.
We state here three of the many results in that paper. Denote the uniform measure on $\Lambda_n$ by
$P^{\text{part}}_n$. Recall that $X_k$ is the number of parts of size $k$ in the random partition.
\begin{theorem}\label{Xk}
If $\lim_{n\to\infty}\frac{k_n}{n^\frac12}=0$, then
$$
\lim_{n\to\infty}P^{\text{part}}_n(\frac{\pi}{\sqrt{6n}}\thinspace k_nX_{k_n}\le x)=1-e^{-x}, \ x\ge0.
$$
\end{theorem}
The above theorem states that for any fixed $k$ ($k_n=k$) as well as for $k_n$ growing at a sufficiently slow rate,
the rescaled quantity $\frac{\pi}{\sqrt{6n}}k_n X_{k_n}$ converges in distribution as $n\to\infty$ to an exponential random variable with parameter 1.
Compare this with
\eqref{nonunifpart1}.

The next theorem concerns $Y_s$, the $s$-th largest part of the random partition.
In the case $s=1$, it was originally proved in \cite{EL}.
\begin{theorem}\label{Ys}
\begin{equation*}
\lim_{n\to\infty}P^{\text{part}}_n(\frac\pi{\sqrt{6n}}Y_s-\log\frac{\sqrt{6n}}\pi\le x)=\int_{-\infty}^x\frac{\exp(-e^{-y}-sy)}{(t-1)!}dy,\ x\in\mathbb{R}, s\in\mathbb{N}.
\end{equation*}
\end{theorem}
Compare Theorem \ref{Ys} to \eqref{nonunifpart2} and \eqref{nonunifpart3}.

If instead of looking at $Y_s$, one considers $Y_{s_n}$ with $s_n\to\infty$ at a sufficiently slow rate, then one gets
convergence in distribution to  a Gaussian
limiting distribution.
\begin{theorem}\label{Ysn}
If $\lim_{n\to\infty}s_n=\infty$ and $\lim_{n\to\infty}\frac{s_n}{n^\frac14}=0$, then
$$
\lim_{n\to\infty}P^{\text{part}}_n(\pi\sqrt{\frac{s_n}{6n}}Y_{s_n}-
\sqrt{s_n}\log\frac{\sqrt{6n}}{\pi s_n}\le x)=\int_{-\infty}^x\frac1{\sqrt{2\pi}}e^{-\frac{y^2}2}dy,\ x\in\mathbb{R}.
$$
\end{theorem}

 We now sketch the method used to prove Theorem \ref{Xk}.
 The proofs of  Theorems \ref{Ys} and \ref{Ysn}  use the same method along with some additional ingredients.
 Let $\Lambda$ denote the set
of all partitions of all  nonnegative integers, including the empty partition of 0.
Let $p(n)=|\Lambda_n|$ denote the number of partitions of $n$.
 As is well-known and easy to check, the generating function $\sum_{n=0}^\infty p(n)q^n$
 for $\{p(n)\}_{n=0}^\infty$ is given by the formula
\begin{equation}\label{genfuncpart}
\sum_{n=0}^\infty p(n)q^n=\prod_{j=1}^\infty(1-q^j)^{-1}.
\end{equation}
 Define $|\lambda|=n$, if $\lambda\in\Lambda_n$.
 For $q\in(0,1)$, consider the following probability measure on the set $\Lambda$ of all partitions of nonnegative integers:
 \begin{equation}\label{Qmeas}
 Q_q(\lambda)=q^{|\lambda|}\prod_{j=1}^\infty(1-q^j),\ \lambda\in\Lambda.
 \end{equation}
 To see that $Q_q$ is a probability measure, note that
 $$
 \sum_{\lambda\in\Lambda_n}Q_q(\lambda)=p(n)q^n\prod_{j=1}^\infty(1-q^j),
 $$
 and thus using \eqref{genfuncpart},
 $$
 \sum_{\lambda\in\Lambda}Q_q(\lambda)=\sum_{n=0}^\infty\sum_{\lambda\in\Lambda_n}Q_q(\lambda)=
 \prod_{j=1}^\infty(1-q^j)\sum_{n=0}^\infty q^np(n)=1.
 $$
 The key observation is that under $Q_q$, the random variables $\{X_j\}_{j=1}^\infty$ are independent
 and the distribution of $X_j$ is geometric with parameter $q^j$.
 To verify the above observation, note that  the values $\{X_j(\lambda)\}_{j=1}^\infty$ uniquely determine $\lambda$, and
note that  $|\lambda|=\sum_{j=1}^\infty jX_j(\lambda)$. Thus, from \eqref{Qmeas},
 $$
 Q_q(X_j=x_j, j=1,2,\cdots)=q^{\sum_{j=1}^\infty jx_j}\prod_{j=1}^\infty(1-q^j)=\prod_{j=1}^\infty (q^j)^{x_j}(1-q^j),
 n=0,1,\cdots.
 $$
The product on the right hand side above indicates the independence of $\{X_j\}_{j=1}^\infty$ under $Q_q$.

 Let $N=|\lambda|$. Of course, $P^{\text{part}}_n(N=n)=1$. However, under $Q_q$, the number $N$ being partitioned is
 random. From \eqref{Qmeas}, we have
$$
Q_q(N=n)=p(n)q^n\prod_{j=1}^\infty(1-q^j),\ n=0,1\cdots.
$$

Note from the definition of the uniform measure $P^{\text{part}}_n$ on $\Lambda_n$ and from the definition of the measure $Q_q$
 that if $\lambda_1$ and $\lambda_2$ are partitions for which
$N(\lambda_1)=N(\lambda_2)=n$, then $P^{\text{part}}_n(\lambda_1)=P^{\text{part}}_n(\lambda_2)$ and
$Q_q(\lambda_1)=Q_q(\lambda_2)$.
From these fact it follows that
\begin{equation}\label{condequ}
P^{\text{part}}_n(A)=Q_q(A|N=n),\  A\subset \Lambda_n.
\end{equation}

Let $E^{Q_q}$ denote expectation with respect to $Q_q$.
Consider  the probability generating function of $N$ under $Q_q$, defined
by $\Phi(r)=E^{Q_q}r^N$. We have
\begin{equation}\label{probgen}
\begin{aligned}
&\Phi(r)=E^{Q_q}r^N=\sum_{n=0}^\infty Q_q(N=n)r^n=\\ &\prod_{j=1}^\infty(1-q^j)\sum_{n=0}^\infty p(n)q^nr^n=\prod_{j=1}^\infty\frac{1-q^j}{1-(qr)^j},
\end{aligned}
\end{equation}
where we have used \eqref{genfuncpart} in the final equality.
Using the fact that $\Phi'(1)=E^{Q_q}N$ and $\Phi''(1)=E^{Q_q}N(N-1)$, one can calculate the expected value
and variance of $N$:
$$
\begin{aligned}
&E^{Q_q}N=\sum_{j=1}^\infty \frac{jq^j}{1-q^j};\\
&\text{Var}_{Q_q}(N)=\sum_{j=1}^\infty \frac{j^2q^j}{(1-q^j)^2}.
\end{aligned}
$$
An asymptotic analysis shows that if one chooses
\begin{equation}\label{qnexp}
q=q_n=e^{-\frac\pi{\sqrt{6n}}},
\end{equation}
then $\text{Var}_{Q_{q_n}}(N)\sim\frac{\sqrt{24}}\pi n^\frac32$
and  $|n-E^{Q_{q_n}}N|=o(n^\frac34)$.
This shows that the random variable $N$ under the measure $Q_{q_n}$ concentrates around $n$. Indeed,
$$
|N-n|\le |n-E^{Q_{q_n}}N|+|N-E^{Q_{q_n}}N|=o(n^{\frac34})+ |N-E^{Q_{q_n}}N|,
$$
and by Chebyshev's inequality,
$$
Q_{q_n}(|N-E^{Q_{q_n}}N|)\ge n^\frac34l_n)\le \frac{O(n^\frac32)}{n^\frac32l_n^2}\to0,\ \text{if}\ l_n\to\infty.
$$

Thus, with probability approaching 1 as $n\to\infty$, $N$ will be within $l_nn^\frac34$ of $n$ if $l_n\to\infty$.
Using  Fourier analysis, it was shown in \cite{Frist} that
$Q_{q_n}(N=n)\sim(96n^3)^{-\frac14}$. From these results, the author was able to show that the
Prohorov distance (a certain metric) between the distribution of $X_{k_n}$ under
the probability measure $Q_{q_n}$ and its distribution under  the measure $Q_{q_n}(\cdot\thinspace|N=n)$ converges to 0 as $n\to\infty$, if $k_n=o(n^\frac12)$. It then follows from
\eqref{condequ} that the Prohorov distance between the distributions of $X_{k_n}$ under  $Q_{q_n}$ and under $P^{\text{part}}_n$ also converges to 0 as $n\to\infty$.
From this, it follows that the $P^{\text{part}}_n$-probability appearing on the left hand side
of Theorem \ref{Xk} can  be replaced by the $Q_{q_n}$ probability.
However, as noted,  under $Q_{q_n}$, the random variables  $\{X_j\}_{j=1}^\infty$ are independent with geometric distributions; thus, everything can
 be calculated explicitly.
 Indeed, we have
$$
\begin{aligned}
&Q_{q_n}(\frac{\pi}{\sqrt{6n}}k_nX_{k_n}\le x)=Q_{q_n}(X_{k_n}\le \frac{\sqrt{6n}x}{\pi k_n})=
\sum_{j=0}^{\lfloor \frac{\sqrt{6n}x}{\pi k_n}\rfloor}(1-q_n^{k_n})(q_n^{k_n})^j=\\
&1-q_n^{k_n\lfloor \frac{\sqrt{6n}x}{\pi k_n}\rfloor}
\stackrel{n\to\infty}{\to}1-e^{-x},
\end{aligned}
$$
where in the final step we use \eqref{qnexp}.
\medskip

\section{Threshold calculations for Erd\H{o}s-R\'enyi random graphs using the first and second moment methods}

We
recall the definition of the  Erd\H{o}s-R\'enyi random graph $G(n,p)$, for $n\in\mathbb{N}$ and
$p\in(0,1)$. This graph has $n$ vertices and thus $\binom n2$ possible edges.
Independently, each of these edges is included in the graph with probability $p$, and not included with probability
$1-p$.
We will analyze  three so-called threshold properties of the random graphs $G(n,p_n)$ as $n\to\infty$, in two cases with variable $p_n$, and in one case with fixed $p_n=p$.
In what follows, $P_n$ and $E_n$ will be used for probabilities and expectations concerning $G(n,p)$
or $G(n,p_n)$.

The method of proof we use in these threshold calculations is based on the  first and second moment methods
that were stated and proved in section \ref{tools}---see \eqref{firstmommeth} and \eqref{secondmommeth}.

\medskip

\noindent \bf  Threshold for disconnected vertices.\rm\
The first result we present concerns
disconnected vertices in $G(n,p_n)$.
\begin{theorem}\label{disconnectvertex}
Let $D_n$ be the random variable denoting the number of disconnected vertices in $G(n,p_n)$.

\noindent i. If $p_n=\frac{\log n+c_n}n$, and $\lim_{n\to\infty}c_n=\infty$, then
$D_n\stackrel{\text{dist}}{\to}0$; equivalently,
$\lim_{n\to\infty}P_n(D_n=0)=1$;

\noindent ii. If $p_n=\frac{\log n+c_n}n$, and $\lim_{n\to\infty}c_n=-\infty$, then
$\frac{D_n}{E_nD_n}\stackrel{\text{dist}}{\to}1$. Also,
$\lim_{n\to\infty}P_n(D_n>M)=1$, for any $M\in\mathbb{N}$.
\end{theorem}
\begin{proof}
For part (i), we use the first moment method.  For $1\le j\le n$, let $D_{n,j}$ denote the indicator random variable for vertex $j$ to be disconnected; that is, $D_{n,j}$  is equal to 1 if vertex $j$  is disconnected
in $G(n,p_n)$,
 and equal to 0 otherwise.
Then we can represent the random variable $D_n$ as $D_n=\sum_{j=1}^n D_{n,j}$.
For distinct vertices $j,k\in[n]$, let  $I_{n,j,k}$ denote the indicator random variable that is equal to 1 if $G(n,p_n)$ contains an edge connecting $j$ and $k$, and is equal to 0 otherwise.
Then $D_{n,j}=\prod_{k\in[n];k\neq j}(1-I_{n,j,k})$.
By the definition of the Erd\H{o}s-R\'enyi graph, the random variables $\{I_{n,j,k}\}_{k\in[n],k\neq j}$ are independent and  $E_nI_{n,j,k}=P_n(I_{n,j,k}=1)=p_n$.
Thus,
\begin{equation}\label{Dnj}
E_nD_{n,j}=\prod_{k\in[n];k\neq j}E(1-I_{n,j,k})=(1-p_n)^{n-1}
\end{equation}
 and
\begin{equation}\label{EDn}
E_nD_n=n(1-p_n)^{n-1}.
\end{equation}
Substituting $p_n$ as in part (i) of the theorem
 into \eqref{EDn}, it follows that
 $\lim_{n\to\infty}E_nD_n=0$, and thus by the first moment method,
 $D_n$ converges in distribution to 0; that is $\lim_{n\to\infty}P_n(D_n\ge\epsilon)=0$, for all $\epsilon>0$. Since $D_n$ is integer-valued, this gives
 $\lim_{n\to\infty}P(D_n=0)=1$.

 We  use the second moment method for part (ii).
 We write
 \begin{equation}\label{Dn2ndmom}
 E_nD_n^2=E\sum_{j=1}^nD_{n,j}\sum_{k=1}^nD_{n,k}=\sum_{j=1}^nE_nD_{n,j}+2\sum_{1\le j<k\le n}E_nD_{n,j}D_{n,k}.
 \end{equation}
 We have
 \begin{equation}\label{pairs}
\begin{aligned}
& E_nD_{n,j}D_{n,k}=P_n(D_{n,j}=D_{n,k}=1)=\\
&P_n(I_{n,j,l}=I_{n,k,m}=0,\ l\in[n]-\{j\}, m\in[n]-\{j,k\})=(1-p_n)^{2n-3},\ j\neq k.
 \end{aligned}
 \end{equation}
From \eqref{Dnj}-\eqref{pairs},  we obtain
\begin{equation}\label{together}
\sigma^2(D_n)=E_nD_n^2-(E_nD_n)^2=n(n-1)(1-p_n)^{2n-3}+n(1-p_n)^{n-1}-\big(n(1-p_n)^{n-1}\big)^2.
\end{equation}
Let $p_n$ be as in part (ii) of the theorem. Then standard estimates give
$$
n(n-1)(1-p_n)^{2n-3}\sim\big(n(1-p_n)^{n-1}\big)^2\sim e^{-2c_n}\stackrel{n\to\infty}{\to}\infty.
$$
Using this with \eqref{together} and \eqref{EDn}, we conclude that
$\sigma^2(D_n)=o((E_nD_n)^2)$ and $\lim_{n\to\infty}E_nD_n=\infty$. In particular then,  by the second moment method,
$\frac{D_n}{E_nD_n}\stackrel{\text{dist}}{\to}1$.  Since $E_nD_n\stackrel{n\to\infty}{\to}\infty$, it follows from the convergence in distribution that
$\lim_{n\to\infty}P_n(D_n>M)=1$, for any $M\in\mathbb{N}$.
\end{proof}

\medskip

\noindent \bf Threshold for connectivity.\rm\
We now investigate the threshold for connectivity in $G(n,p_n)$.
If $G(n,p_n)$ has at least one disconnected vertex, then of course $G(n,p_n)$ is disconnected.
Thus, the threshold  of $p_n$ for connectivity is greater than or equal to that for no disconnected vertices.
It turns out that the threshold for connectivity   is asymptotically the same as for no disconnected vertices.
The following result was first proved in \cite{ER}.
\begin{theorem}\label{connectivity}

\noindent i. If $p_n=\frac{\log n+c_n}n$, and $\lim_{n\to\infty}c_n=\infty$, then

\noindent $\lim_{n\to\infty}P_n(G(n,p_n)\ \text{\rm is connected})=1$;

\noindent ii. If $p_n=\frac{\log n+c_n}n$, and $\lim_{n\to\infty}c_n=-\infty$, then

\noindent $\lim_{n\to\infty}P_n(G(n,p_n)\ \text{\rm is connected})=0$.
\end{theorem}
The proof of part (ii) of course follows from part (ii) of Theorem \ref{disconnectvertex}, so we only need  to prove part (i).
The proof is not via the simple   first moment method presented above, but rather via a slightly more involved first moment technique \cite{Boll}.

\noindent \it Proof of part (i).\rm\ We may and will assume that $n\ge3$ and $1\le c_n\le \log n$.
For each $A\subset[n]$, let $I_{A,n}$ be the indicator random variable that is equal to one if in $G(n,p_n)$ every vertex in $A$ is   disconnected from the set of vertices $[n]-A$, and that is equal to zero otherwise.
Then
\begin{equation}\label{ProbIA}
P_n(I_{A,n}=1)=(1-p_n)^{k(n-k)},\ \text{if}\ |A|=k.
\end{equation}
Clearly,
\begin{equation}\label{unionequal}
P_n(G(n,p_n)\ \text{\rm is disconnected})= P_n(\cup_{A\subset[n],|A|\le \lfloor\frac n2\rfloor}\{I_{A,n}=1\}).
\end{equation}
Since the probability of the union of events is less than or equal to the sums of the probabilities of the individual events composing the union, we have from
\eqref{ProbIA} and \eqref{unionequal} that
\begin{equation}\label{unionbd}
P_n(G(n,p_n)\ \text{\rm is disconnected})\le\sum_{k=1}^{\lfloor\frac n2\rfloor}\binom nk(1-p_n)^{ k(n-k)}.
\end{equation}

We break the sum on the right hand side of \eqref{unionbd} into two parts.
Substituting for $p_n$ from part (i) of the theorem, using the estimate
$\binom nk\le (\frac{en}k)^k$ from Stirling's formula,  using
 the inequality $1-x\le e^{-x}$, for $x\ge0$, and recalling our additional assumption  $1\le c_n\le \log n$ from
the first line of the proof,  we have
\begin{equation*}\label{firstpart}
\begin{aligned}
&\sum_{k=1}^{\lfloor n^\frac34\rfloor}\binom nk(1-p_n)^{ k(n-k)}\le\sum_{k=1}^{\lfloor n^\frac34\rfloor}(\frac{en}k)^ke^{-k(\log n+c_n)}e^{2k^2\frac{\log n}n}=\\
& e^{(1-c_n)}e^{\frac{2\log n}n}+\sum_{k=2}^{\lfloor n^\frac34\rfloor}
\big(\frac{e^{1-c_n}e^{2k\frac{\log n}n}}k\big)^k\le\\
& e^{(1-c_n)}e^{\frac{2\log n}n}+e^{2(1-c_n)}\sum_{k=2}^{\lfloor n^\frac34\rfloor}\big(\frac{e^{2n^{-\frac14}\log n}}2\big)^k.
\end{aligned}
\end{equation*}
 Consequently, for sufficiently large $n$,
 \begin{equation}\label{firstpartagain}
 \sum_{k=1}^{\lfloor n^\frac34\rfloor}\binom nk(1-p_n)^{ k(n-k)}\le e^2e^{-c_n}+e^{2(1-c_n)}.
 \end{equation}

 For the second part of the sum, using two of the inequalities used for the first part, using
 the fact that $k(n-k)\ge \frac{kn}2$ (which of course holds over the entire original range of $k$), and substituting
 for $p_n$, we have
 \begin{equation}\label{secondpart}
 \begin{aligned}
 &\sum_{k=\lfloor n^\frac34\rfloor+1}^n\binom nk(1-p_n)^{ k(n-k)}\le\sum_{k=\lfloor n^\frac34\rfloor+1}^n(\frac{en}k)^ke^{-\frac{k\log n}2}\le\\
 &\sum_{k=\lfloor n^\frac34\rfloor+1}^n(en^{\frac14})^kn^{-\frac k2}=\sum_{k=\lfloor n^\frac34\rfloor+1}^n(en^{-\frac14})^k.
 \end{aligned}
 \end{equation}
 (Note that for this part of the sum we have ignored the term $c_n$ in $p_n$ as it is not needed.)
 From \eqref{unionbd}-\eqref{secondpart}, we conclude that
 $$
 \lim_{n\to\infty}P_n(G(n,p_n)\ \text{is disconnected})=0,
 $$
 if $p_n$ is as in part (i) of the theorem.\hfill $\square$
\medskip

\noindent \bf Threshold for clique size.\rm\
We now consider a problem regarding $G(n,p)$ with $p\in(0,1)$ fixed.
 Recall that a \it clique \rm of size $k$ in $G(n,p)$ is
a complete subgraph on $k$ vertices, and that the \it clique number \rm of the graph is the largest $k$ for which a clique of size $k$
exists. Let $\log_{\frac1p}$ denote the logarithm in base $\frac1p$ and let $\log^{(2)}_\frac1p=\log_{\frac1p}\log_\frac1p$.
The following result reveals the close to deterministic nature of the clique number  in $G(n,p)$.
See \cite{Boll,Pin14}.
\begin{theorem}\label{clique}
Let $L_n$ be the random variable denoting the clique number of $G(n,p)$. Then
$$
\lim_{n\to\infty}P_n(L_n\ge 2\log_{\frac1p}n- c\log^{(2)}_{\frac1p} n)=\begin{cases} 0,\ \text{if}\ c<2;\\ 1,\ \text{if}\ c>2.\end{cases}
$$
\end{theorem}
 Theorem \ref{clique}
is an immediate corollary of the following theorem \cite{Pin14}.
 For each $k\in[n]$, define the random variable
$N_{n,k}$ to be the \it number\rm\ of cliques in $G(n,p)$ of size $k$.
\begin{theorem}
\noindent i. If $k_n\ge 2\log_{\frac1p}n-c\log^{(2)}_\frac1pn$, for some $c<2$, then
\begin{equation}\label{Exp0}
\lim_{n\to\infty}E_nN_{n,k_n}=0.
\end{equation}
Thus,
$N_{n,k_n}\stackrel{\text{dist}}{\to}0$; equivalently, $\lim_{n\to\infty}P_n(N_{n,k_n}=0)=1$;

\noindent ii. If $k_n\le 2\log_{\frac1p}n-c\log^{(2)}_\frac1pn$, for some $c>2$,
then
\begin{equation}\label{Exp1}
\lim_{n\to\infty}E_nN_{n,k_n}=\infty.
\end{equation}
Also,
\begin{equation}\label{wlln}
\frac{N_{n,k_n}}{EN_{n,k_n}}\stackrel{\text{dist}}{\to}1,
\end{equation}
and
\begin{equation}\label{wtoinfty}
\lim_{n\to\infty}P_n(N_{n,k_n}>M)=1,\ \text{for any}\ M\in\mathbb{N}.
\end{equation}
\end{theorem}
Note that the claims in the sentence following \eqref{Exp0} are a
consequence of the first moment method along with the fact that $N_{n,k_n}$ is integer-valued.
Note also that \eqref{wtoinfty} follows immediately from \eqref{Exp1} and
\eqref{wlln}.
We will give the proofs of \eqref{Exp0} and \eqref{Exp1}; they are of course
first moment calculations.
The proof of \eqref{wlln} can be given by the second moment method,
but the calculations are much more involved than they were in the case of Theorem \ref{disconnectvertex}.
We will begin the second moment calculation so as to reveal the considerations that arise, and then refer the interested reader
elsewhere for  a complete proof.
\medskip

\noindent \it Proofs of \eqref{Exp0} and \eqref{Exp1}.\rm\
We will prove the two results simultaneously. For \eqref{Exp0},
fix $c<2$ and for each $n$ sufficiently large, let $c_n\le c$ be such that
$k_n:= 2\log_{\frac1p}n-c_n\log^{(2)}_\frac1pn$ is an integer.
For \eqref{Exp1}, fix $c>2$ and for each $n$ sufficiently large, let $c_n\ge c$ be
such that $k_n:= 2\log_{\frac1p}n-c_n\log^{(2)}_\frac1pn$ is an integer.
We need to show that \eqref{Exp0} holds
for the first choice of $k_n$ and that \eqref{Exp1} holds for the  second choice of $k_n$.

There are $\binom n{k_n}$ different subsets of $[n]$ of size $k_n$. For each such subset
$K$, let $I_K$ be the indicator random variable that is equal to 1 if the vertices in $K$ form a clique
in $G(n,p)$, and equal to 0 otherwise.
Then
$$
N_{n,k_n}=\sum_{K\subset[n],|K|=k_n}1_K.
$$
From the definition of the Erd\H{o}s-R\'enyi graph, we have
\begin{equation}\label{EK}
E_n1_K=p^{\binom{k_n}2},\ |K|=k_n.
\end{equation}
Thus,
$$
E_nN_{n,k_n}=\binom n{k_n}p^{\binom{k_n}2}.
$$
It is not hard to show that as long as $k_n=o(n^\frac12)$, then
\begin{equation}\label{approx-standard}
\binom n{k_n}\sim\frac{n^{k_n}}{k_n!}.
\end{equation}
From \eqref{EK} and \eqref{approx-standard} along  with Stirling's formula, we obtain
\begin{equation}\label{EnNnkn}
E_nN_{n,k_n}\sim\frac{n^{k_n}p^{\frac{k_n(k_n-1)}2}}{k_n^{k_n}e^{-k_n}\sqrt{2\pi k_n}}.
\end{equation}

Taking logarithms in base $\frac1p$ on both sides of \eqref{EnNnkn}, we have
\begin{equation}\label{afterlog}
\log_\frac1pE_nN_{n,k_n}\sim k_n\log_\frac1pn-\frac12k_n^2+\frac12k_n-k_n\log_\frac1pk_n+k_n\log_\frac1pe-\frac12\log_\frac1p2\pi k_n.
\end{equation}
We have
\begin{equation}\label{log1pkn}
\begin{aligned}
&\log_\frac1pk_n=\log_\frac1p\big(2\log_\frac1pn-c_n\log^{(2)}_\frac1pn\big)=
\log_\frac1p\Big((\log_\frac1pn)\big(2-\frac{c_n\log^{(2)}_\frac1pn}{\log_\frac1pn}\big)\Big)=\\
&\log^{(2)}_\frac1pn+\log_\frac1p\big(2-\frac{c_n\log^{(2)}_\frac1pn}{\log_\frac1pn}\big)=
\log^{(2)}_\frac1pn+O(1).
\end{aligned}
\end{equation}
The sum of the three terms of highest orders on the right hand side of \eqref{afterlog} is
$k_n\log_\frac1pn-\frac12k_n^2-k_n\log_\frac1pk_n$.
Using \eqref{log1pkn} to substitute for $\log_\frac1pk_n$, and using the definition of $k_n$ to substitute
for $k_n$ where it appears elsewhere in this three term sum,
we have
\begin{equation}\label{threeterms}
\begin{aligned}
&k_n\log_\frac1pn-\frac12k_n^2-k_n\log_\frac1pk_n=(2\log_\frac1pn-c_n\log^{(2)}_\frac1pn)\log_\frac1pn-\\
&\frac12\big(2\log_\frac1pn-c_n\log^{(2)}_\frac1pn\big)^2-(2\log_\frac1pn-c_n\log^{(2)}_\frac1pn)
(\log^{(2)}_\frac1pn+O(1))=\\
&(c_n-2)(\log_\frac1pn)\log^{(2)}_\frac1pn+O(\log_\frac1pn).
\end{aligned}
\end{equation}
The sum of the rest of the terms on the right hand side of \eqref{afterlog} satisfies
\begin{equation}\label{otherthreeterms}
\frac12k_n+k_n\log_\frac1pe-\frac12\log_\frac1p2\pi k_n=O(\log_\frac1pn).
\end{equation}

From \eqref{afterlog}, \eqref{threeterms} and \eqref{otherthreeterms},
in the case of  the first choice of $k_n$, for which $c_n\le c<2$, we have
$$
\lim_{n\to\infty}\log_\frac1pE_nN_{n,k_n}=-\infty,
$$
and thus, \eqref{Exp0} holds,
while in the case of the second choice of $k_n$, for which $c_n\ge c>2$, we have
$$
\lim_{n\to\infty}\log_\frac1pE_nN_{n,k_n}=\infty,
$$
and thus \eqref{Exp1} holds.
\hfill $\square$
\medskip

We now begin the second moment calculation to give an idea of what  is involved. Since $\sigma^2(N_{n,k_n})=E_nN^2_{n,k_n}-(E_nN_{n,k_n})^2$,
the requirement $\sigma^2(N_{n,k_n})=o((E_nN_{n,k_n})^2)$ that is needed for the  second moment method can be written
as
$$
E_n(N_{n,k_n})^2=(E_nN_{n,k_n})^2+o\big((E_nN_{n,k_n})^2\big).
$$
We label the subsets $K\subset[n]$ satisfying $|K|=k_n$ according to the vertices they contain.
Thus $K_{i_1,\cdots, i_{k_n}}$ denotes the subset $\{i_1,\cdots, i_{k_n}\}$.
Let $I_{K_{i_1,\cdots, i_{k_n}}}$ denote the indicator random variable that is equal to one if
the vertices in $K_{i_1,\cdots, i_{k_n}}$ form a clique in $G(n,p)$, and equal to zero elsewhere.
Then we can  write $N_{n,k_n}$ in the form
$$
N_{n,k_n}=\sum_{1\le i_1<i_2<\cdots<i_{k_n}\le n}I_{K_{i_1,\cdots, i_{k_n}}}.
$$
Consequently,
$$
E_nN^2_{n,k_n}=\sum_{\substack{1\le i_1<i_2<\cdots<i_{k_n}\le n\\1\le l_1<l_2<\cdots<l_{k_n}\le n}}E_n
I_{K_{i_1,\cdots, i_{k_n}}}I_{K_{l_1,\cdots, l_{k_n}}}.
$$
Note that $E_nI_{K_{i_1,\cdots, i_{k_n}}}I_{K_{l_1,\cdots, l_{k_n}}}$ equals 1 or 0 depending on
whether or not the graph $G(n,p)$ possesses an edge between every pair
of vertices in $K_{i_1,\cdots, i_{k_n}}$ and between every pair of vertices in $K_{l_1,\cdots, l_{k_n}}$.
It is easy to see that the expression
$\sum_{1\le l_1<l_2<\cdots<l_{k_n}\le n}E_n
I_{K_{i_1,\cdots, i_{k_n}}}I_{K_{l_1,\cdots, l_{k_n}}}$ is independent of the particular set
$K_{i_1,\cdots, i_{k_n}}$. Thus, choosing the set $K_{1,\cdots, k_n}$, we have
$$
E_nN^2_{n,k_n}=\binom n{k_n}
\sum_{1\le l_1<l_2<\cdots<l_{k_n}\le n}E_n
I_{K_{1,\cdots, k_n}}I_{K_{l_1,\cdots, l_{k_n}}}.
$$

Let $J=J(l_1,\cdots, l_{k_n})=|K_{1,\cdots, k_n}\cap K_{l_1,\cdots, l_{k_n}}|$ denote the number
of vertices shared by $K_{1,\cdots, k_n}$ and $K_{l_1,\cdots, l_{k_n}}$.
It is not hard to show that
$$
E_n
I_{K_{1,\cdots, k_n}}I_{K_{l_1,\cdots, l_{k_n}}}=\begin{cases}
p^{2\binom {k_n}2-\binom J2},\ \text{if}\ J=J(l_1,\cdots, l_{k_n})\ge2;\\
p^{2\binom {k_n}2},\ \text{if}\ J=J(l_1,\cdots, l_{k_n})\le 1.\end{cases}.
$$
Thus, we can write
$$
E_nN^2_{n,k_n}=\binom n{k_n}\sum_{\substack{1\le l_1<l_2<\cdots<l_{k_n}\le n\\ J(l_1,\cdots, l_{k_n})\le 1}}p^{2\binom{k_n}2}+
\binom n{k_n}\sum_{\substack{1\le l_1<l_2<\cdots<l_{k_n}\le n\\ J(l_1,\cdots, l_{k_n})\ge 2}}p^{2\binom{k_n}2-\binom J2}.
$$
It turns out that the first term on the right hand side above is equal to
$(E_nN_{n,k_n})^2+o\big((E_nN_{n,k_n})^2\big)$, while the second term is equal to
$o\big((E_nN_{n,k_n})^2\big)$. However the proofs of these statements require considerable additional calculation.
We refer the reader to \cite{Pin14}, and note that the notation there is a little different from the notation here.

\medskip

A deeper and more difficult result than the above ones is the identification and analysis of
a striking phase transition in the connectivity properties of $G(n,p_n)$ between the cases
$p=\frac cn$ with $c\in(0,1)$ and $p_n=\frac cn$ with $c>1$.
Note that by
Theorem \ref{connectivity}, for such values of $p_n$ the probability of  $G(n,p_n)$ being connected approaches zero
as $n\to\infty$. The phase transition that occurs for $p_n$ close to $\frac1n$ concerns the size of the largest
connected component, which makes a transition from logarithmic to linear as a function of $n$.
We only state the result, which was proved by
Erd\H{o}s and R\'enyi in 1960 \cite{ER}. For more recent proofs, see \cite{AS,Pin14}.
\begin{theorem}\label{ErRen}
\noindent i. Let $p_n=\frac cn$, with $c\in(0,1)$. Then there exists
a $\gamma=\gamma(c)$ such that the size $C_n^{\text{lg}}$ of the largest connected component  of $G(n,p_n)$
satisfies
$$
\lim_{n\to\infty}P_n(C_n^{\text{lg}}\le \gamma\log n)=1;
$$
\noindent ii. Let $p=\frac cn$ with $c>1$. Then there exists a unique solution
$\beta=\beta(c)\in(0,1)$ to the equation $1-e^{-cx}-x=0$. For any $\epsilon>0$, the size
$C_n^{\text{lg}}$ of the largest connected component of $G(n,p_n)$ satisfies
$$
\lim_{n\to\infty}P((1-\epsilon)\beta n\le C_n^{\text{lg}}\le (1+\epsilon)\beta n)=1.
$$
\end{theorem}
A detailed analysis of the largest connected component in the case that $p_n=\frac1n+o(\frac1n)$ can be found
in \cite{Boll84}.

\section{Arcsine law for random walks and a combinatorial lemma of Sparre-Andersen}
Consider the simple symmetric random walk (SSRW) $\{S_n^\pm\}_{n=0}^\infty$ on $\mathbb{Z}$.
It is constructed from a sequence $\{X^{\pm}_n\}_{n=1}^\infty$ of IID  random variables with
the following Bernoulli distribution with parameter $\frac12$: $P(X^{\pm}_n=1)=P(X^{\pm}_n=-1)=\frac12$.
One defines $S_0^\pm=0$ and $S_n^\pm=\sum_{j=1}^n X^{\pm}_j$, \ $n\ge1$.
It is well-known that the SSRW is \it recurrent \cite{Durrett}, \cite{Pin14}\rm:
$$
P(\liminf_{n\to\infty} S_n^\pm=-\infty\ \text{and}\ \limsup_{n\to\infty}S_n^\pm=+\infty)=1.
$$
Let
$$
A_{+,n}=A_{+,n}(\{S_j^\pm\}_{j=1}^n)=|\{j\in[n]: S_j^\pm>0\}|
$$
denote   the number of steps  up until step $n$ that the SSRW spends on the  positive half axis.
If one wants to consider things completely symmetrically, then one can define instead
$$
 A^{\text{sym}}_{+,n}= A^{\text{sym}}_{+,n}(\{S_j^\pm\}_{j=1}^n)=|\{j\in[n]: S_j^\pm>0\ \text{or}\ S_j^\pm=0\ \text{and}\ S_{j-1}^\pm>0\}|.
$$
There are several ways, some more probabilistic and some more combinatoric, to prove the following result.
\begin{theorem}
For SSRW,
\begin{equation}\label{arcsinn}
P( A^{\text{\rm sym}}_{+,2n}=2k)=\frac{\binom{2k}k\binom{2n-2k}{n-k}}{2^{2n}},\ k=0,1,\cdots, n.
\end{equation}
\end{theorem}
\noindent See, for example, \cite{Feller1} for a somewhat probabilistic approach and see \cite{Pin14} for a completely combinatorial approach using Dyck paths.
Stirling's formula
gives
$$
\frac{\binom{2k}k\binom{2n-2k}{n-k}}{2^{2n}}\sim\frac1\pi\frac1{\sqrt{k(n-k)}},
$$
and for any $\epsilon\in(0,1)$,  the above estimate is uniform for $k\in[\epsilon n,(1-\epsilon)n]$ as $n\to\infty$.
Using this estimate  with \eqref{arcsinn}, a  straightforward calculation \cite{Pin14} reveals that
\begin{equation}\label{arcsinlim}
\lim_{n\to\infty}P(\frac1{2n}A^{\text{\rm sym}}_{+,2n}\le x)=\frac2\pi\arcsin\sqrt x, \ x\in[0,1].
\end{equation}
Of course it then follows that \eqref{arcsinlim} also holds with $n$ in place of $2n$.
It is not hard to show that the proportion of time up to $n$ that the SSRW is equal to zero converges in distribution to 0.
See, for example, \cite[Exercise 4.3]{Pin14}.
Thus,
 we also have
 \begin{equation}\label{arcsinlimagain}
\lim_{n\to\infty}P(\frac1nA_{+,n}\le x)=\frac2\pi\arcsin\sqrt x, \ x\in[0,1].
\end{equation}
The distribution whose distribution function is $\frac2\pi\arcsin\sqrt x$, for  $x\in[0,1]$, is called the \it Arcsine distribution.\rm\
 The density of this distribution is given by
 $\frac{d}{dx}\big(\frac2\pi\arcsin\sqrt x)\big)=\frac1\pi\frac1{\sqrt{x(1-x)}}$.
We  note that in \eqref{arcsinn}, the most likely values of the distribution are $k=0$ and $k=n$, and the least likely values are $k=\frac n2$, if $n$ is even,  and $k=\lfloor \frac n2\rfloor\pm1$, if $n$ is odd.
 In the limit, the resulting  density is unbounded at $0^+$ and $1^-$, and attains its minimum  at $x=\frac12$.
 Thus,  counter-intuitively, it turns out that   the most likely proportions of time that  the SSRW is on the positive half axis are 0 and 1, and the least likely proportion
 is $\frac12$. For a discussion of this, see \cite{Feller1}.

 Now consider  a general symmetric random walk on $\mathbb{R}$. It is constructed from an IID sequence of random variables $\{X_n\}_{n=1}^\infty$ with an arbitrary
 symmetric distribution: $P(X_n\ge x)=P(X_n\le -x)$, for $x\ge0$.
  We define $\mathcal{S}_0=0$ and $\mathcal{S}_n=\sum_{j=1}^nX_j,\ n\ge1$.
 We will assume that the distribution is continuous; that is
 $P(X_n=x)=0$, for all $x\in\mathbb{R}$. It then follows that $P(\mathcal{S}_n=x)=0$, for all $x\in\mathbb{R}$
 and $n\in\mathbb{N}$.
 %Therefore $A^{\text{sym}}_{+,n}(\{\mathcal{S}_j\}_{j=1}^n)=A_{+,n}(\{\mathcal{S}_j\}_{j=1}^n)$ a.s.
 In the sequel we will  use the fact that $P(\mathcal{S}_n>0)=1-P(\mathcal{S}_n<0)$.
Here is a truly remarkable result concerning $A_{+,n}=A_{+,n}(\{\mathcal{S}_j\}_{j=1}^n)$:
\begin{theorem}\label{Sparre}
For any symmetric random walk $\{\mathcal{S}_n\}_{n=0}^\infty$ generated from a continuous distribution,
\begin{equation}\label{SAndersen}
P(A_{+,n}=k)=\frac{\binom{2k}k\binom{2n-2k}{n-k}}{2^{2n}},\ k=0,1,\cdots, n.
\end{equation}
\end{theorem}
Thus, \eqref{arcsinlimagain} not only holds for the SSRW on $\mathbb{Z}$, but also for  general symmetric random walks on $\mathbb{R}$.
 Theorem \ref{Sparre} was first proved by Sparre-Andersen in 1949 \cite{Sparre-A}. The very complicated original proof has since been very much simplified by combinatorial arguments.
 This simplification  is an example par excellence of the use of combinatorial tools to prove probabilistic results.
  Our exposition is a streamlined version of the proof in \cite{Feller2}. For a somewhat different approach, see
  \cite{Kallenberg02}.

A combinatorial lemma will show that Theorem \ref{Sparre} is equivalent to the following result, of independent interest.
\begin{theorem}\label{Ln}
Let $\{\mathcal{S}_n\}_{n=0}^\infty$ be a symmetric random walk on $\mathbb{R}$ generated by a continuous distribution.
Let $L_n$ denote the index (that is, location) of the first maximum of $\{\mathcal{S}_0,\mathcal{S}_1,\cdots, \mathcal{S}_n)$. (Actually, by the continuity assumption, the maximum almost surely
occurs at only one location.)
Then
\begin{equation}\label{SAndersenagain}
P(L_n=k)=\frac{\binom{2k}k\binom{2n-2k}{n-k}}{2^{2n}},\ k=0,1,\cdots, n.
\end{equation}
\end{theorem}
\bf \noindent Remark.\rm\ Thus,
$$
\lim_{n\to\infty}P(\frac1nL_n\le x)=\frac2\pi\arcsin\sqrt x, \ x\in[0,1].
$$

\noindent \it Proof of Theorem \ref{Ln}.\rm\ We write the event
$\{L_n=k\}$
as
$$
\{L_n=k\}=A_k\cap B_{n-k}, \ k=0,\cdots, n,
$$
where
$$
A_k=\{\mathcal{S}_k>\mathcal{S}_0,\cdots, \mathcal{S}_k>\mathcal{S}_{k-1}\}, \ k=1,\cdots, n
$$
and
$$
B_{n-k}=\{\mathcal{S}_k\ge\mathcal{S}_{k+1},\cdots, \mathcal{S}_k\ge\mathcal{S}_n\},\ k=0,\cdots, n-1,
$$
and $A_0$ and $B_0$ are equal to the entire probability space, so $P(A_0)=P(B_0)=1$.
The events $A_k$ and $B_{n-k}$ are independent since
$A_k$ depends only on $\{X_j\}_{j=1}^k$ and $B_{n-k}$ depends only on $\{X_j\}_{j=k+1}^n$.
Thus,
\begin{equation}\label{ABL}
P(L_n=k)=P(A_k)P(B_{n-k}),\ k=0,\cdots, n.
\end{equation}
From the equidistribution of the $\{X_n\}_{n=1}^\infty$ and the continuity and symmetry of the distribution,
 it follows that
\begin{equation}\label{q}
\begin{aligned}
&P(A_k)=q_k,\ k=1,\cdots, n;\\
&P(B_{n-k})=q_{n-k},\  k=0,\cdots, n-1,
\end{aligned}
\end{equation}
where
\begin{equation}\label{qn}
q_n=P(\mathcal{S}_1>0,\cdots, \mathcal{S}_n>0),\ n=1,2,\cdots.
\end{equation}

Define $q_0=1$ and let
 $H(s)$ be the generating function for $\{q_n\}_{n=0}^\infty$:
$$
H(s)=\sum_{n=0}^\infty q_ns^n.
$$
From \eqref{ABL} and \eqref{q} we have
$$
1=\sum_{k=0}^nP(L_n=k)=\sum_{k=0}^nq_kq_{n-k},
$$
from which we obtain
$H^2(s)=(1-s)^{-1}$.
Thus,
$$
H(s)=(1-s)^{-\frac12}.
$$
One has
$$
\frac{d^nH}{ds^n}(0)=\frac{(2n)!}{n!2^{2n}},\ n\ge0.
$$
Thus,
\begin{equation}\label{qnform}
q_n=\frac1{2^{2n}}\binom {2n}n,\ n\ge0.
\end{equation}
 Theorem \ref{Ln} follows from \eqref{ABL}, \eqref{q} and \eqref{qnform}.\hfill $\square$

\medskip

From  \eqref{qn} and\eqref{qnform} we obtain the
following  corollary of interest, which is not needed for the proof of
Theorem \ref{Sparre}.
\begin{corollary}\label{corrsparre}
\begin{equation}
P(\mathcal{S}_1>0,\cdots, \mathcal{S}_n>0)=\frac1{2^{2n}}\binom {2n}n,\ n\ge1.
\end{equation}
\end{corollary}

We now turn to the  combinatorial lemma.
Let $n\in\mathbb{N}$, let
 $\{x_j\}_{j=1}^n$ be real numbers,  denote their partial sums by $\{s_j\}_{j=1}^n$,
 and let $s_0=0$.
 We say that the first maximum among the partial sums occurs at $j\in\{0,\cdots, n\}$ if
 $s_j>s_i$, for $i=0,\cdots, j-1$, and $s_j\ge s_i$, for $i=j+1,\cdots, n$.
 For each permutation $\sigma\in S_n$, let $s^\sigma_0=0$ and
  $s^\sigma_j=\sum_{i=1}^jx_{\sigma_i}$, \  $j=1,\cdots, n$,
 denote the partial sums of the permuted sequence $\{x_{\sigma_j}\}_{j=1}^n$.
\begin{Clemma}
\it Let $r$ be an integer satisfying $0\le r\le n$. The number $A_r$ of permutations $\sigma\in S_n$ for
which exactly $r$ of the partial sums are strictly positive is equal to the number $B_r$  of permutations
$\sigma\in S_n$ for which the first maximum among these partial sums occurs at
 position $r$.
\end{Clemma}\rm\
\begin{proof}
We prove the lemma by induction. The result is true for $n=1$. Indeed, if $x_1>0$, then $A_1=B_1=1$ and
$A_0=B_0=0$, while if $x_1\le 0$, then $A_1=B_1=0$ and $A_0=B_0=1$.  Now let $n\ge2$ and assume that the result is true for
$n-1$. Denote by $A^{(k)}_r$ and $B^{(k)}_r$ the values corresponding to $A_r$ and $B_r$ when the $n$-tuple
$(x_1,\cdots, x_n)$ is replaced by the $(n-1)$-tuple obtained by deleting $x_k$. By the inductive hypothesis,
$A^{(k)}_r=B^{(k)}_r$, for $k=1,\cdots, n$ and $r=0,\cdots, n-1$. This is also true for $r=n$ since trivially
$A^{(k)}_n=B^{(k)}_n=0$.

We break up the rest of the proof into two cases. The first case is when $\sum_{k=1}^nx_k\le 0$.
We construct the $n!$ permutations of $(x_1,\cdots, x_n)$ by first selecting $k\in[n]$ and placing $x_k$ in the
last position, and then permuting the remaining $n-1$ numbers in the first $n-1$ positions. Since the
$n$th partial sum for every permutation is equal to $\sum_{k=1}^nx_k$, which is non-positive, it is clear that
the number of positive partial sums and the index of the first maximal partial sum depend only on the numbers
in the first $n-1$ positions.  Thus, $A_r=\sum_{k=1}^nA^{(k)}_r$ and $B_r=\sum_{k=1}^n B^{(k)}_r$, and consequently,
$A_r=B_r$ by the  inductive hypothesis.

The second case is when $\sum_{k=1}^nx_k>0$. In this case, the $n$th partial sum is positive, and thus the previous
argument shows that $A_r=\sum_{k=1}^n A^{(k)}_{r-1}$, for $r=1,\cdots, n$, and $A_0=0$.
To obtain an analogous formula for $B_r$, construct the $n!$ permutations of $(x_1,\cdots, n_n)$ by first
selecting $k\in[n]$ and placing $x_k$ in
the first position, and then permuting the remaining $n-1$ numbers in the last $n-1$ positions. Such
a permutation is of the form $(x_k,x_{j_1},\cdots, x_{j_{n-1}})$.
Since the $n$th partial sum is positive for every permutation, it follows that $s^\sigma_0=0$ is not a maximal partial sum for any permutation $\sigma$, so $B_0=0$. Thus, $A_0=B_0$. Clearly, the first maximal partial sum for the
permuted order $(x_k,x_{j_1},\cdots, x_{j_{n-1}})$ occurs at index $r\in[n]$ if and only if the first maximum of the partial sums of $(x_{j_1},\cdots, x_{j_{n-1}})$ occurs at $r-1$. Thus, $B_r=\sum_{k=1}^nB^{(k)}_{r-1}$,
for $r=1,\cdots, n$.  By the inductive hypothesis, it then follows that $A_r=B_r$.
\end{proof}
We now prove  Theorem \ref{Sparre} using  the  Combinatorial Lemma and  Theorem
\ref{Ln}.

\noindent \it Proof of Theorem \ref{Sparre}.\rm\
For $n\in\mathbb{N}$, label the permutations in $S_n$ from 1 to $n!$, with the first one being the identity permutation.
Consider the random variables $\{X_k\}_{k=1}^n$ and their partial sums $\{\mathcal{S}_k\}_{k=1}^n$.
Denote the partial sums for the $j$th permutation of $\{X_k\}_{k=1}^n$
 by $\{\mathcal{S}^{(j)}_k\}_{k=1}^n$; thus,
$\mathcal{S}^{(1)}_k=\mathcal{S}_k$.
For a fixed integer $r$ satisfying $0\le r\le n$, let $Z^{(j)}_{n,r}$ be the indicator random variable
equal to 1 if permutation number $j$ has exactly $r$ positive sums, and equal to 0 otherwise.
By symmetry, the random variables $\{Z^{(j)}_{n,r}\}_{j=1}^{n!}$ have the same distribution.
Thus, we have
\begin{equation}\label{Zs}
P(A_{+,n}=r)=P(Z^{(1)}_{n,r}=1)=EZ^{(1)}_{n,r}=
\frac1{n!}E\sum_{k=1}^n Z^{(k)}_{n,r}.
\end{equation}
Similarly, let $W^{(j)}_{n,r}$ be the indicator random variable
equal to 1 if for permutation number $j$, the first maximal partial sum has index $r$,  and equal to 0 otherwise.
By symmetry, the random variables $\{W^{(j)}_{n,r}\}_{j=1}^{n!}$ have the same distribution.
Thus, we have
\begin{equation}\label{Ws}
\begin{aligned}
P(L_n=r)=P(W^{(1)}_{n,r}=1)=EW^{(1)}_{n,r}=\frac1{n!}E\sum_{k=1}^n W^{(k)}_{n,r},
\end{aligned}
\end{equation}
where $L_n$ is as in Theorem \ref{Ln}.
By the Combinatorial Lemma,
$\sum_{k=1}^n Z^{(k)}_{n,r}=\sum_{k=1}^n W^{(k)}_{n,r}$, for every realization
of $\{X_j\}_{j=1}^n$. Thus, the right hand sides of \eqref{Zs} and \eqref{Ws} are equal.
Thus $P(A_{+,n}=r)=P(L_n=r), \ r=0,\cdots, n$, which in conjunction with Theorem \ref{Ln} proves
Theorem \ref{Sparre}.\hfill $\square$
\medskip

The material in this section is one of the cornerstones of a larger subject which includes
 path decomposition and the Wiener-Hopf factorization  for random walks and for L\'evy processes.
 See \cite{Feller1} for some basic material on Wiener-Hopf factorization. For more modern and more comprehensive
 work, see for example \cite{Doney,Kypri,PitmanTang} and references therein.
The exploitation of symmetries, of which the proofs of Theorems \ref{Sparre} and \ref{Ln}   are  examples, is an important recurring theme in probability theory; see for example \cite{Kallenberg05}.

\section{The probabilistic method}

I cannot imagine a survey of the bridge between combinatorics and probability that does not include
the probabilistic method,  popularized  by Erd\H{o}s beginning  in the 1940s.
The method gives only a one-sided bound, and that bound is  usually not tight,
yet frequently   it is the only known method to yield a non-trivial result.
Almost every combinatorist is familiar with Erd\H{o}s' use of the method to get a lower bound
on the Ramsey number $R(s,s)$.
A small tweaking of the most basic use of the method (\cite{AS}, \cite{Pin14}) gives the lower bound
$R(k,k)\ge(1+o(1))\frac k{\sqrt2e}2^{\frac k2}$, leaving a wide gap with the best known upper bound
which is around $4^k$.
We  illustrate the probabilistic method with two other examples, one entirely straightforward
and simplistic, the other a bit more creative.
For much more on the probabilistic method, see \cite{AS}.
\medskip

\noindent \bf Two-colorings with  no monochromatic sets.\rm\
Consider $N$ sets of objects, not necessarily disjoint,
 each of size $k$, with $N,k\ge2$. Then the total number of distinct objects in the union of the sets is at least $k$ and at most $Nk$.
 Given $k$, how large can $N$ be so that no matter what the configuration of the objects, it is always possible to assign one of two colors to each object in such a way
 that none of the   $N$ sets is monochromatic?

\noindent \it Example.\rm\ Let $N=3$ and $k=2$. Consider three objects, labelled 1,2 and 3, and define  three sets
of size two by $\{1,2\},\{1,3\},\{2,3\}$. Then every two-coloring will produce a monochromatic set.

\begin{proposition}
Consider $N$ sets of not necessarily disjoint objects, where each set contains $k$ objects. If
$$
N(\frac12)^{k-1}<1,
$$
then it is always possible to choose a two-coloring of the objects in such a way that none of the $N$ sets is monochromatic.
\end{proposition}
\begin{proof}
Call the two colors black and white.
To use  the probabilistic method, one colors all the objects independently and at random so that each object
is colored  black or white with equal probabilities $\frac12$.
Let $A_j$ be the event that the $j$th set is monochromatic. Then
$P(A_j)=(\frac12)^k+(\frac12)^k=(\frac12)^{k-1}$.
Let $A=\cup_{j=1}^NA_j$ be the event that at least one of the $N$ sets is monochromatic.
Then $P(A)\le \sum_{j=1}^NP(A_j)=N(\frac12)^{k-1}$.
 If $N(\frac12)^{k-1}<1$, then $P(A)<1$, which guarantees that $A^c\neq\emptyset$.
 \end{proof}

 \medskip
\noindent \bf Maximal antichains and Sperner's Theorem.\rm\
Let $n\in\mathbb{N}$. Recall that a family $\mathcal{F}_n$ of subsets of $[n]$ is called an
\it antichain\rm\ in $[n]$ if no set belonging to  $\mathcal{F}_n$ is contained in another set belonging
 to  $\mathcal{F}_n$.
 \begin{theorem}\label{anti}
 Let $\mathcal{F}_n$ be an antichain in $[n]$. Then
\begin{equation}\label{antichain}
 \sum_{A\in\mathcal{F}_n}\frac1{\binom n{|A|}}\le 1.
 \end{equation}
 \end{theorem}
 As a corollary of Theorem \ref{anti}, we obtain \it Sperner's Theorem.
 \begin{corollary}
Let $\mathcal{F}_n$ be an antichain in $[n]$. Then
\begin{equation}
 |\mathcal{F}_n|\le\binom n{\lfloor \frac n2\rfloor}.
\end{equation}
 \end{corollary}
\noindent \it Proof of Corollary.\rm\ The function $a\to\binom na$ is maximized at $a=\lfloor\frac n2\rfloor$.
Thus, from \eqref{antichain}, we
have
$|\mathcal{F}_n|\frac1{\binom n{\lfloor \frac n2\rfloor}}\le 1$.
\hfill $\square$

Sperner's theorem originally appeared in \cite{Sperner}. See \cite{AZG} for a short non-probabilistic proof.

\medskip

\noindent \it Proof of Theorem \ref{anti}.\rm\
We follow the argument in \cite{AS}. Fix an antichain $\mathcal{F}_n$ in $[n]$.
For each permutation $\sigma\in S_n$, define the family  $\mathcal{C}_\sigma$ of subsets of $[n]$
by
$$
\mathcal{C}_\sigma=\big\{\{\sigma_i:1\le i\le j\}:1\le j\le n\big\}.
$$
Let $X_n(\sigma)=|\mathcal{F}_n\cap\mathcal{C}_\sigma|$ denote the number of subsets of $[n]$ common to
both $\mathcal{F}_n$ and $\mathcal{C}_\sigma$.
We now consider $S_n$ with the uniform probability measure $P_n$; this turns $X_n=X_n(\sigma)$
into a real-valued random variable and $\mathcal{C}=\mathcal{C}_\sigma$ into a family-of-sets-valued random variable
on the probability space $(S_n,P_n)$. We represent $X_n$ as a sum of indicator random variables.
For $A\in\mathcal{F}_n$, let $I_A$ equal 1 if $A\in\mathcal{C}$ and 0 otherwise. Then
$$
X_n=\sum_{A\in\mathcal{F}_n}I_A
$$
and
\begin{equation}\label{expectindic}
E_nX_n=\sum_{A\in\mathcal{F}_n}E_nI_A=\sum_{A\in\mathcal{F}_n}P_n(A\in\mathcal{C}).
\end{equation}
The random family $\mathcal{C}$ of subsets of $[n]$ contains exactly one subset of size $|A|$.
Since $\sigma$ is a uniformly random permutation, this one subset of size $|A|$ is distributed uniformly over all subsets of size $|A|$. Thus,
\begin{equation}\label{indicatorprob}
P_n(A\in\mathcal{C})=\frac1{\binom n{|A|}}.
\end{equation}
By construction, for any $\sigma\in S_n$, the family of subsets $\mathcal{C}_\sigma$ forms a \it chain\rm; that is
for $A,B\in\mathcal{C}_\sigma$, either $A\subset B$ or $B\subset A$. Thus, since $\mathcal{F}_n$ is an antichain, it follows
that $X_n(\sigma)=|\mathcal{F}_n\cap\mathcal{C}_\sigma|\le 1$, for all $\sigma\in S_n$. In particular then,
\begin{equation}\label{exple1}
E_nX_n\le 1.
\end{equation}
The theorem follows from \eqref{expectindic}-\eqref{exple1}.\hfill $\square$

\medskip

\bf \noindent Acknowledgement.\rm\ I thank Jim Pitman, as this article was much enhanced through  the generous advice and constructive criticism he  proffered upon viewing  a  first draft and then again upon viewing a second draft.

\end{document}